\documentclass[12pt,letterpaper]{article}

\usepackage{amsmath,amssymb,amsthm,amsfonts,amstext}
\usepackage{etoolbox}
\usepackage{lscape}
\usepackage[longnamesfirst]{natbib}
\usepackage{bm}
\usepackage{tabularx}
\usepackage{longtable}
\usepackage[flushleft]{threeparttable}
\appto\TPTnoteSettings{\linespread{1}\footnotesize} 
\makeatletter 
   \g@addto@macro\TPT@defaults{\linespread{1}\footnotesize} 
\makeatother
\usepackage{booktabs}
\usepackage{siunitx}
\usepackage{color}
\usepackage{graphicx}
\usepackage{caption}
\usepackage{subcaption}
\usepackage{mathrsfs}
\usepackage{setspace}
\usepackage{rotating}
\usepackage{multirow}
\usepackage{array}
\newcolumntype{$}{>{\global\let\currentrowstyle\empty}}
\newcolumntype{^}{>{\currentrowstyle}}
\newcolumntype{C}{%
  >{\rowstyle{}}c%
}
\newcolumntype{B}{%
  >{\currentrowstyle}S[detect-weight]%
}
\newcommand{\rowstyle}[1]{\gdef\currentrowstyle{#1}%
  #1\ignorespaces
}
\usepackage{thrmappendix}
\usepackage{alltt}
\usepackage{hypernat}
\usepackage{mathtools,dsfont}
\usepackage{hyperref}
\usepackage[capitalize,noabbrev]{cleveref}

\newtoggle{SUPPLEMENTAL}\toggletrue{SUPPLEMENTAL}
\togglefalse{SUPPLEMENTAL} 

\newtoggle{BLINDED}\toggletrue{BLINDED}
\togglefalse{BLINDED} 

\numberwithin{equation}{section}

\theoremstyle{plain}
\newtheorem{assumption}{Assumption}

\newtheorem{theorem}{Theorem}[section]
\newtheorem{lemma}[theorem]{Lemma}
\newtheorem{proposition}[theorem]{Proposition}

\theoremstyle{definition}

\newcommand{\pconv}{\xrightarrow{p}}
\newcommand{\dconv}{\xrightarrow{d}}

\newcommand{\tr}{^{\top}}

\crefformat{footnote}{#2\footnotemark[#1]#3} 
\crefname{conjecture}{Conjecture}{Conjectures}
\crefname{section}{Section}{Sections}
\crefname{subsection}{Section}{Sections}
\crefname{subsubsection}{Section}{Sections}
\Crefname{conjecture}{Conjecture}{Conjectures}
\Crefname{section}{Section}{Sections}
\Crefname{subsection}{Section}{Sections}
\Crefname{subsubsection}{Section}{Sections}
\crefname{appendix}{Appendix}{Appendices}
\crefname{subappendix}{Appendix}{Appendices}
\crefname{subsubappendix}{Appendix}{Appendices}
\Crefname{appendix}{Appendix}{Appendices}
\Crefname{subappendix}{Appendix}{Appendices}
\Crefname{subsubappendix}{Appendix}{Appendices}
\crefname{equation}{}{}
\Crefname{equation}{Equation}{Equations}
\crefformat{enumi}{(#2#1#3)}
\crefrangeformat{enumi}{(#3#1#4)\crefrangeconjunction(#5#2#6)}
\crefmultiformat{enumi}{(#2#1#3)}{ and~(#2#1#3)}{, (#2#1#3)}{ and~(#2#1#3)}
\Crefformat{enumi}{Part (#2#1#3)}
\Crefrangeformat{enumi}{Parts (#3#1#4)\crefrangeconjunction(#5#2#6)}
\Crefmultiformat{enumi}{Parts (#2#1#3)}{ and~(#2#1#3)}{, (#2#1#3)}{ and~(#2#1#3)}
\crefname{assumption}{}{}
\Crefname{assumption}{Assumption}{Assumptions}
\newcommand{\crefrangeconjunction}{--}

\DeclareGraphicsExtensions{.pdf,.png}
\graphicspath{ {./Figures/} }

\newcommand{\citeposs}[1]{\citeauthor{#1}'s (\citeyear{#1})}
\newcommand{\numnornd}[1]{\num[round-mode=off,group-digits=integer]{#1}} 

\newcommand{\matf}[1]{#1} 
\newcommand{\vecf}[1]{#1} 
\DeclareMathOperator{\Var}{Var}

\DeclareMathOperator{\Corr}{Corr}

\newcommand{\Varp}[1]{\Var\left(#1\right)}

\newcommand{\R}{{\mathbb R}}

\DeclareMathOperator{\E}{E}
\DeclareMathOperator{\Q}{Q}
\let\Pr\relax \DeclareMathOperator{\Pr}{P} 

\DeclareMathOperator*{\argmin}{arg\,min}

\DeclareMathOperator{\1}{\mathds{1}}
\newcommand{\Ind}[1]{\1\left\{#1\right\}}
\newcommand{\Normal}{\mathrm{N}}
\newcommand{\Normalp}[2]{\Normal\left(#1,#2\right)}

\newcommand{\UnifDist}{\textrm{Unif}}

\newcommand{\diff}{d} 
\newcommand{\pD}[2]{\frac{\partial #1}{\partial #2}}
 
\newcommand{\D}[2]{\frac{\diff #1}{\diff #2}}
 
\newcommand{\independenT}[2]{\mathrel{\rlap{$#1#2$}\mkern2mu{#1#2}}}
\newcommand\independent{\protect\mathpalette{\protect\independenT}{\perp}} 

\providecommand{\absbig}[1]{\left\lvert#1\right\rvert}
\providecommand{\normbig}[1]{\left\lVert#1\right\rVert}
\providecommand{\abs}[1]{\lvert#1\rvert}
\providecommand{\norm}[1]{\lVert#1\rVert}
      
\let\originalleft\left
\let\originalright\right
\renewcommand{\left}{\mathopen{}\mathclose\bgroup\originalleft}
\renewcommand{\right}{\aftergroup\egroup\originalright}

  \renewenvironment{thebibliography}[1]{%
    \begin{oldthebibliography}{#1}%
      \setlength{\parskip}{0ex}%
      \setlength{\itemsep}{0ex}%
  }%
  {%
    \end{oldthebibliography}%
  }

\allowdisplaybreaks[3]

\title{Smoothed {GMM} for quantile models\footnote{The authors would like to express their appreciation to 
Ron Gallant, Duke Kao, Carlos Lamarche, Xiaodong Liu, Carlos Martins-Filho, Eric Renault, Pedro Sant'Anna, and other 
seminar and conference participants at the University of Illinois at Urbana--Champaign, Pennsylvania State University, University of Iowa, Brown University, University of Connecticut, University of Colorado (Boulder), 2017 Midwest Econometrics Group, and the 2017 North American Summer Meeting of the Econometric Society, 
for helpful comments, references, and discussions. Code is provided for all methods, simulations, and applications.}}

\makeatletter
\def\thisisjusttotrickthesyntaxcheck{\begin{tabular}}
\def\and{\end{tabular}%
  \hskip 0.9em \@plus.17fil\relax  
  \begin{tabular}[t]{c}}
\def\thisisalsojusttotrickthesyntaxcheck{\end{tabular}}
\makeatother

\author{Luciano de Castro\thanks{Department of Economics, University of Iowa. E-mail: \texttt{decastro.luciano@gmail.com}}
\and
Antonio F.\ Galvao\thanks{Department of Economics, University of Arizona. E-mail: \texttt{agalvao@email.arizona.edu}}
\and
David M.\ Kaplan\thanks{Department of Economics, University of Missouri. E-mail: \texttt{kaplandm@missouri.edu}}
\and
Xin Liu\thanks{Department of Economics, University of Missouri. E-mail: \texttt{xl6f6@mail.missouri.edu}}
}

\begin{document}

\maketitle

\begin{abstract}
This paper develops theory for feasible estimators of finite-dimensional parameters identified by general conditional quantile restrictions, under much weaker assumptions than previously seen in the literature. 
This includes instrumental variables nonlinear quantile regression as a special case. 
More specifically, we consider a set of unconditional moments implied by the conditional quantile restrictions, providing conditions for local identification. 
Since estimators based on the sample moments are generally impossible to compute numerically in practice, we study feasible estimators based on smoothed sample moments. 
We propose a method of moments estimator for exactly identified models, as well as a generalized method of moments estimator for over-identified models. 
We establish consistency and asymptotic normality of both estimators under general conditions that allow for weakly dependent data and nonlinear structural models. 
Simulations illustrate the finite-sample properties of the methods. 
Our in-depth empirical application concerns the consumption Euler equation derived from quantile utility maximization. 
Advantages of the quantile Euler equation include robustness to fat tails, decoupling of risk attitude from the elasticity of intertemporal substitution, and log-linearization without any approximation error. 
For the four countries we examine, the quantile estimates of discount factor and elasticity of intertemporal substitution are economically reasonable for a range of quantiles above the median, even when two-stage least squares estimates are not reasonable.

\vspace*{0.5\baselineskip}



\end{abstract}

\thispagestyle{empty}

\newpage

\baselineskip19.5pt
\pagenumbering{arabic}

\section{Introduction}

Since the seminal work of \citet{KoenkerBassett78}, quantile regression (QR) has attracted considerable interest in statistics and econometrics. 
QR estimates conditional quantile functions that provide insight into heterogeneous effects of policy variables. 
This is especially valuable for program evaluation studies, where these methods help analyze how treatments or social programs affect the outcome's distribution. 
Nevertheless, endogeneity has been a pervasive concern in economics due to simultaneous causality, omitted variables, measurement error, self-selection, and estimation of equilibrium conditions, among other causes. 
Extending the standard QR, \citet{ChernozhukovHansen05, ChernozhukovHansen06, ChernozhukovHansen08} present results on identification, estimation, and inference for an instrumental variables QR (IVQR) model that allows for endogenous regressors.\footnote{We refer to \citet{ChernozhukovHansenWuthrich17} for an overview of IVQR. They discuss alternative, complementary QR models with endogeneity in Section 1.2.5, specifically the triangular system and local quantile treatment effect (LQTE) model. Even in the LQTE model, the IVQR estimator has a meaningful interpretation; see \citet{Wuthrich16b}.}
However, computational difficulties have limited practical estimators to linear models with iid data (discussed below).

Under weaker conditions than prior IVQR papers, we develop theory around feasible smoothed estimators.%
\footnote{The methods developed in this paper are also related to those for semiparametric and nonparametric models. Identification, estimation, and inference of general (non-smooth) conditional moment restriction models have received much attention in the econometrics literature, as in \citet[\S7]{NeweyMcFadden94}, \citet{ChenLintonvanKeilegom03}, \citet{ChenPouzo09,ChenPouzo12}, and \citet{ChenLiao15}, for example. 
However, theoretical results are only for unsmoothed estimators that are often not computationally feasible in practice.} 
We consider a set of unconditional moments implied by a general parametric conditional quantile restriction and study exactly identified and over-identified models. 
Under misspecification of the conditional model, our results still hold for the pseudo-true parameter solving the unconditional moments, complementing the results in \citet{AngristChernozhukovFernandez-Val06} for QR. 
For identification, we provide sufficient conditions for local identification based on these moments. 
For estimation, since using unsmoothed sample moments is generally intractable, we study smoothed estimators that compute quickly and may have improved precision \citep{KaplanSun17}. 
Specifically, we develop smoothed method of moments (MM) and smoothed generalized method of moments (GMM) quantile estimators for exactly identified and over-identified models, respectively.\footnote{QR and IVQR have been discussed as GMM by \citet[\S III.A]{Buchinsky98} and \citet[\S1.3.1]{ChernozhukovHansenWuthrich17}, among others.} 
Unlike prior IVQR estimation papers, we allow for weakly dependent data and nonlinear structural models when establishing the large sample properties of the estimators, namely, consistency and asymptotic normality.

Our in-depth empirical study estimates a quantile Euler equation using aggregate time series data. 
This equation is derived from a quantile utility maximization model. This model is an interesting alternative to the standard expected utility model because it is robust to fat tails, allows heterogeneity through the quantiles, decouples the elasticity of intertemporal substitution from risk attitude, and results in an Euler equation that does not suffer from any approximation error when log-linearized.\footnote{Heavy tails in consumption data have been documented recently by \citet{TodaWalsh15,TodaWalsh17}.} 
Quantile preferences were first studied by \citet{Manski88} and were axiomatized by \citet{Chambers09} and \citet{Rostek10}. 
\Citet{deCastroGalvao17} use  quantile preferences in a dynamic economic setting and provide a comprehensive analysis of a dynamic rational quantile model. 
They derive the policy function (Euler equation) as a nonlinear conditional quantile restriction. 
Consequently, we may use smoothed GMM to estimate the structural parameters, including the elasticity of intertemporal substitution (EIS). 
Numerous papers have estimated the EIS, e.g., \citet{HansenSingleton83}, \citet{Hall88}, \citet{CampbellMankiw89}, \citet{OgakiReinhart98}, and \citet{Yogo04}. 
For the four countries we study, the smoothed quantile estimates of the discount factor and EIS are economically reasonable for a range of quantiles above the median, including cases where the 2SLS estimates are not reasonable.

For IVQR estimation of the model in \citet{ChernozhukovHansen05}, the literature lacks results for feasible estimators allowing nonlinear structural models and dependent data.\footnote{For nonlinear QR (no IV), see \citet[\S2.2]{Powell94}, \citet{OberhoferHaupt15}, and references therein.} 
The following are iid sampling assumptions: 
Condition (i) on p.\ 310 in \citet{ChernozhukovHong03}, 
Assumption 2.R1 in \citet{ChernozhukovHansen06}, and 
Assumption 1 in \citet{KaplanSun17}. 
A nonlinear structural model is allowed by the computationally demanding\footnote{\Citet{ChernozhukovHansenWuthrich17} comment, ``This approach bypasses the need to optimize a non-convex and non-smooth criterion at the cost of needing to design a sampler that adequately explores the quasi-posterior in a reasonable amount of computation time.''} Markov Chain Monte Carlo estimator in \citet[Ex.\ 3, p.\ 297ff.]{ChernozhukovHong03}, but linear-in-parameters models are required in 
(3.4) in \citet{ChernozhukovHansen06} and 
Assumption 1 in \citet{KaplanSun17}. 
Additionally, 
\citet{ChernozhukovHansen06} note, ``The computational advantages of our estimator rapidly diminish as the number of endogenous variables increases'' (p.\ 501). 
Even if only one observed variable is endogenous, this restriction limits the use of interactions and transformations (like polynomial terms). 
The results from \citet{ChernozhukovHansen06} have been extended in unpublished work by \citet{SuYang11} to non-iid data for use with a correctly specified linear spatial autoregressive model, treating regressors and instruments as nonstochastic. 
\Citet{ChenLee17} propose an estimator for linear-in-parameters IVQR models using mixed integer quadratic programming, but computation is very slow: with only four parameters and $n=100$ observations, their Table 1 shows average computation times for IV median regression exceeding five minutes. 
\Citet{Wuthrich17a} proposes an estimator without assuming linearity but only for a binary treatment (and iid data). 
From a Bayesian perspective, \citet{LancasterJun10} allow nonlinear models but only iid data (and require computation over a grid of coefficient values or else by Markov Chain Monte Carlo). 
We relax both iid sampling and linearity in our formal results, while maintaining the computational simplicity, speed, and scalability of the method in \citet{KaplanSun17}, and adding the efficiency of two-step GMM.

Historically, smoothing IVQR moment conditions was proposed first in unpublished notes by \citet{MaCurdyHong99}, mentioned later in (also unpublished) \citet[\S2.4]{MaCurdyTimmins01} and the handbook chapter by \citet[\S5]{MaCurdy07}. 
\citet{Whang06} and \citet{Otsu08} use moment smoothing for empirical likelihood QR. 
The related idea of smoothing non-differentiable \emph{objective functions} goes back to \citet[\S3]{Amemiya82}, if not earlier, and has been used for QR by \citet{Horowitz98}, \citet{GalvaoKato16}, and \cite{FernandesGuerreHorta17}, among others. 

\Cref{sec:ID} presents the model and discusses identification. 
\Cref{sec:est} develops the smoothed MM and GMM estimators, whose asymptotic properties are provided in \cref{sec:asy}. 
\Cref{sec:sim} contains simulation results. 
In \cref{sec:EIS} we illustrate the new approach empirically. 
\Cref{sec:conclusion} suggests directions for future research. 
The appendix collects all proofs.

We conclude this introduction with some remarks about the notation. 
Random variables and vectors are uppercase ($Y$, $X$, etc.), while non-random values are lowercase ($y$, $x$); 
for vector/matrix multiplication, all vectors are treated as column vectors. 
Also, $\Ind{\cdot}$ is the indicator function, 
$\E(\cdot)$ expectation, 
$\Q_\tau(\cdot)$ the $\tau$-quantile, 
$\Pr(\cdot)$ probability, 
and $\Normalp{\mu}{\sigma^2}$ the normal distribution. 
For vectors, $\normbig{\cdot}$ is the Euclidean norm. 
Acronyms used include those for 
central limit theorem (CLT), 
continuous mapping theorem (CMT), 
elasticity of intertemporal substitution (EIS), 
generalized method of moments (GMM), 
mean value theorem (MVT), 
probability density function (PDF), 
uniform law of large numbers (ULLN), and 
weak law of large numbers (WLLN).

\section{Model and identification}
\label{sec:ID} 

We consider the following nonlinear conditional quantile model
\begin{equation}\label{eq:qrmodel}
\Q_{\tau} [ \Lambda\bigl(\vecf{Y}_i,\vecf{X}_i,\vecf{\beta}_{0\tau}\bigr) \mid \vecf{Z}_i] = 0,
\end{equation}
where 
$\vecf{Y}_i\in\mathcal{Y}\subseteq\R^{d_Y}$ is the endogenous variable vector, 
$\vecf{Z}_i\in\mathcal{Z}\subseteq\R^{d_Z}$ is the full instrument vector that contains $\vecf{X}_i\in\mathcal{X}\subseteq\R^{d_X}$ as a subset, 
$\Lambda(\cdot)$ is the ``residual function'' that is known up to the finite-dimensional parameter of interest $\vecf{\beta}_{0\tau} \in \mathcal{B} \subseteq \R^{d_\beta}$, 
and $\tau \in (0,1)$ is the quantile index. 
The model in \eqref{eq:qrmodel} can be represented by conditional moment restrictions as
\begin{equation}\label{eq:qrgmmmodel}
\vecf{0} = 
\E \left[ \Ind{\Lambda\bigl(\vecf{Y}_i,\vecf{X}_i,\vecf{\beta}_{0\tau}\bigr) \le 0} - \tau \mid \vecf{Z}_i \right], 
\end{equation}
where $\Ind{\cdot}$ is the indicator function.

To estimate $\vecf{\beta}_{0\tau}$, we use unconditional moments implied by \cref{eq:qrgmmmodel}:
\begin{equation}\label{eqn:moments}
\vecf{0} = 
\E\left\{ \vecf{Z}_i \left[ \Ind{\Lambda\bigl(\vecf{Y}_i,\vecf{X}_i,\vecf{\beta}_{0\tau}\bigr) \le 0} - \tau \right] \right\} .
\end{equation}
\Citet{KaplanSun17} consider a special case of \cref{eqn:moments} with $\Lambda\bigl(\vecf{Y},\vecf{X},\vecf{\beta}\bigr)=Y_1-\vecf{Y}_{-1}\tr \vecf{\beta}_1-\vecf{X}\tr \vecf{\beta}_2$, where $Y_1$ is the outcome and $\vecf{Y}_{-1}=(Y_2,\ldots,Y_{d_Y})$ are endogenous regressors. 
We take $\vecf{Z}_i$ as given; see \citet[\S2.1, pp.\ 108--110]{KaplanSun17} for discussion of optimal instruments. 
Our asymptotic results assume only \cref{eqn:moments}, so they are robust to misspecification of the structural model in \cref{eqn:moments}, treating $\vecf{\beta}_{0\tau}$ as the pseudo-true parameter satisfying \cref{eqn:moments}.

Given \cref{eqn:moments}, $\vecf{\beta}_{0\tau}$ is ``locally identified'' if there exists a neighborhood of $\vecf{\beta}_{0\tau}$ within which only $\vecf{\beta}_{0\tau}$ satisfies \cref{eqn:moments}. 
This holds if the partial derivative matrix of the right-hand side of \cref{eqn:moments} with respect to the $\vecf{\beta}$ argument is full rank; see, e.g., \citet[p.\ 787]{ChenChernozhukovLeeNewey14}. 
This full rank condition is formally stated below in \Cref{a:U}(ii). 
The following proposition states the local identification result.

\begin{proposition}\label{prop:local-ID}
Given \cref{eqn:moments} and (the full rank) \Cref{a:U}(ii), $\vecf{\beta}_{0\tau}$ is locally identified.
\end{proposition}

Global identification is notoriously more difficult to establish, 
although \citet[Thm.\ 2 and App.\ C]{ChernozhukovHansen05} provide some results for IVQR.

To fix ideas, we discuss two examples of structural models in the form of \cref{eq:qrmodel}. 
%
The first example is a random coefficient model as in \citet{ChernozhukovHansen05,ChernozhukovHansen06}. 
Let $D$ be an endogenous ``treatment'' (like education), $U\sim\UnifDist(0,1)$ an unobserved variable (like ability), $\vecf{X}$ exogenous regressors, and $\tilde{Y}_d = q\bigl( d, x, \beta_0(U) \bigr)$ potential outcomes (like wage), where $q(\cdot)$ is known and $\vecf{\beta}_0(U)$ is a random coefficient vector depending on $U$. 
Let $\vecf{\beta}_{0\tau} = \vecf{\beta}_0(\tau)$, $Y = (\tilde{Y}_D, D)$, and $\Lambda(\vecf{Y},\vecf{X},\vecf{\beta}) = \tilde{Y}_D - q(D,\vecf{X},\vecf{\beta})$. 
Under their Assumptions A1--A5, Theorem 1 in \citet{ChernozhukovHansen05} yields a conditional quantile restriction like in \cref{eq:qrmodel}: 
\begin{equation*}
\tau 
= \Pr\left[ \Lambda\bigl(\vecf{Y},\vecf{X},\vecf{\beta}_{0\tau}\bigr) \le 0 \mid \vecf{Z} \right] 
.
\end{equation*}

Another example of a structural model applies to our empirical application in \cref{sec:EIS}. 
Under certain assumptions, if individuals maximize the $\tau$-quantile of utility instead of expected utility, then the resulting consumption Euler equation can be written in the form 
\begin{equation*}
\Q_\tau[ \beta_{0\tau} (1+r_{t+1}) U'(C_{t+1})/U'(C_t) \mid \Omega_t ] = 1 , 
\end{equation*}
where 
$\beta$ is the discount factor, 
$r_t$ is real interest rate, 
$C_t$ is consumption, 
$U(\cdot)$ is the utility function, 
$\Omega_{t}$ is the information set, 
and $\Q_{\tau}[ W_t \mid \Omega_{t}]$ denotes the conditional $\tau$-quantile of $W_t$ given $\Omega_{t}$. 
With isoelastic utility and instruments $\vecf{Z}_t$ chosen from $\Omega_t$, we obtain \cref{eq:qrmodel}:
\begin{equation*}
\Q_\tau[ \beta_{0\tau} (1+r_{t+1}) (C_{t+1}/C_t)^{-\gamma_{0\tau}} - 1 \mid \vecf{Z}_t ] = 0 . 
\end{equation*}

\section{The smoothed {MM} and {GMM} estimators}
\label{sec:est}

This section presents smoothed estimators based on the moment conditions in \cref{eqn:moments}. 
The smoothed MM and smoothed GMM estimators are designed for exactly identified and over-identified models, respectively. 
We now introduce notation, followed by the estimators.

Let the population map $\vecf{M} \colon \mathcal{B} \times \mathcal{T} \mapsto \R^{d_Z}$ be 
\begin{align}\label{eqn:def-M}
\vecf{M}(\vecf{\beta},\tau) 
&\equiv \E\left[ \vecf{g}^u_i(\vecf{\beta},\tau) \right] 
,\\ \label{eqn:def-gu}
\vecf{g}^u_i(\vecf{\beta},\tau)
&\equiv \vecf{g}^u(\vecf{Y}_i,\vecf{X}_i,\vecf{Z}_i,\vecf{\beta},\tau)
 \equiv \vecf{Z}_i\left[ \Ind{\Lambda\left(\vecf{Y}_i,\vecf{X}_i,\vecf{\beta}\right) \le 0} - \tau \right] 
, 
\end{align}
where superscript ``$u$'' denotes ``unsmoothed.'' 
The population moment condition \cref{eqn:moments} is 
\begin{equation}\label{eqn:def-beta0}
\vecf{0} = \vecf{M}(\vecf{\beta}_{0\tau},\tau) . 
\end{equation}

Without smoothing, the corresponding sample moments simply replace population expectation $\E(\cdot)$ with sample expectation $\hat{\E}(\cdot)$, i.e., the sample average. 
Analogous to the population map $M(\cdot)$ in \cref{eqn:def-M}, the unsmoothed sample map is 
\begin{equation}\label{eqn:def-Mu-hat}
\hat{\vecf{M}}_n^u\left(\vecf{\beta},\tau\right)
\equiv \hat{\E}\left[ \vecf{g}^u(\vecf{Y},\vecf{X},\vecf{Z},\vecf{\beta},\tau) \right] 
\equiv \frac{1}{n}\sum_{i=1}^{n} \vecf{g}^u_i(\vecf{\beta},\tau) 
.
\end{equation}
The well-known computational difficulty \citep[e.g.,][Fig.\ 1(a) and Ex.\ 3, p.\ 297]{ChernozhukovHong03} of minimizing a GMM criterion based on $\hat{\vecf{M}}_n^u( \vecf{\beta}, \tau )$ comes from the discontinuous indicator function $\Ind{\cdot}$ inside $\vecf{g}^u_i(\vecf{\beta},\tau)$. 
To address this difficulty, we smooth the indicator function.

With smoothing (no ``$u$'' superscript), the sample analogs of \cref{eqn:def-M,eqn:def-gu} are 
\begin{align}
\begin{split}\label{eqn:def-M-hat}
\vecf{g}_{ni}(\vecf{\beta},\tau)
&\equiv \vecf{g}_n(\vecf{Y}_i,\vecf{X}_i,\vecf{Z}_i,\vecf{\beta},\tau)
 \equiv \vecf{Z}_i \bigl[ \tilde{I}\bigl(-\Lambda(\vecf{Y}_i,\vecf{X}_i,\vecf{\beta}) / h_n \bigr) - \tau \bigr] ,  \\
\hat{\vecf{M}}_n\left(\vecf{\beta},\tau\right)
&\equiv \frac{1}{n}\sum_{i=1}^{n} \vecf{g}_{ni}(\vecf{\beta},\tau) , 
\end{split}
\end{align}
where 
$h_n$ is a bandwidth (sequence) and 
$\tilde{I}(\cdot)$ is a smoothed version of the indicator function $\Ind{\cdot \ge 0}$. 
The $\tilde{I}(\cdot)$ in \cref{fig:Itilde} has been used by \citet{Horowitz98}, \citet{Whang06}, and \citet{KaplanSun17}, who use the fact that its derivative is a fourth-order kernel to establish higher-order improvements in the linear iid setting. 
The double subscript on $\vecf{g}_{ni}$ is a reminder that we have a triangular array setup because $\vecf{g}_{ni}$ depends on the bandwidth sequence $h_n$ in addition to 
$(\vecf{Y}_i,\vecf{X}_i,\vecf{Z}_i)$.

\begin{figure}[htbp]
\centering
\includegraphics[width=0.6\textwidth,clip=true,trim=35 35 20 70]{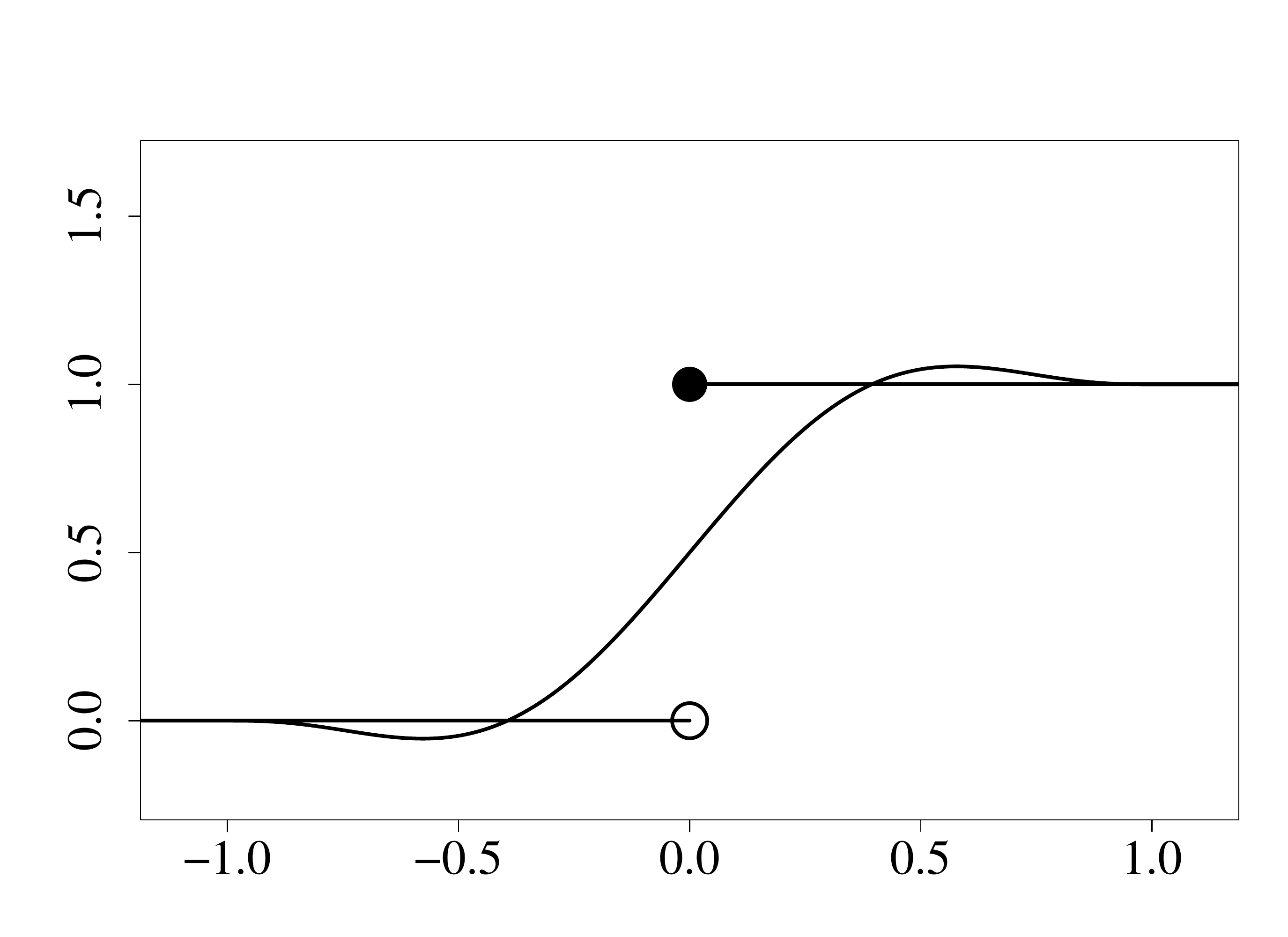}%
\caption{\label{fig:Itilde}$\Ind{u \ge 0}$ and $\tilde{I}(u)=\Ind{-1 \le u \le 1}\bigl[ 0.5+\frac{105}{64}\bigl( u-\frac{5}{3}u^{3}+\frac{7}{5}u^{5}-\frac{3}{7}u^{7}\bigr) \bigr] + \Ind{u>1}$.}
\end{figure}

\subsection{Method of moments (exact identification)}
\label{sec:est-MM}

With exact identification ($d_Z=d_\beta$), our estimator solves the smoothed sample moment conditions\footnote{The right-hand side of \cref{eqn:def-est-MM} can be relaxed to $o_p(n^{-1/2})$.} 
\begin{equation}\label{eqn:def-est-MM}
\hat{\vecf{M}}_n(\hat{\vecf{\beta}}_{\mathrm{MM}},\tau) 
= \vecf{0} .
\end{equation}
Numerically, as long as the bandwidth is not too near zero and $\Lambda(\cdot)$ is differentiable in $\vecf{\beta}$, \cref{eqn:def-est-MM} is easy to solve since the Jacobian exists. 
Further, it is easy to check whether the estimate indeed satisfies \cref{eqn:def-est-MM}. 
In contrast, with over-identification, it is impossible to know if the numerical solution is the global (not just local) minimum of the GMM criterion function.\footnote{See footnote 5 in \citet{ChernozhukovHong03}.} 
Consequently, combining moments and using \cref{eqn:def-est-MM} provides a reliable initial value for the GMM minimization.

\subsection{``One-step'' {GMM} (over-identification)}
\label{sec:est-1s}

With over-identification ($d_Z>d_\beta$), \cref{eqn:def-est-MM} has no solution. 
Thus, a natural GMM estimator is the ``one-step'' estimator proposed in \citet[p.\ 2151]{NeweyMcFadden94} that takes one Newton--Raphson-type step from an initial consistent (but not efficient) estimator. 
\Citet[Thm.\ 3.5, p.\ 2151]{NeweyMcFadden94} show that this is sufficient for asymptotic efficiency, although they assume $g(\cdot)$ is smooth and fixed. 
We use the one-step estimator only for intermediate computation, focusing on the more common two-step estimator (in \cref{sec:est-GMM}) for asymptotic theory.

Let
\begin{equation} \label{eqn:G-bar}
\bar{\matf{G}}=\frac{1}{n} \sum_{i=1}^{n} \vecf{\nabla}_{\vecf{\beta}\tr} \vecf{g}_{ni}(\bar{\vecf{ \beta}},\tau) , 
\end{equation}
where $\vecf{\nabla}_{\vecf{\beta}\tr}$ denotes the partial derivative with respect to  $\vecf{\beta}\tr$, and 
$\bar{\vecf{\beta}}$ 
is an initial estimator 
consistent for 
$\vecf{\beta}_{0\tau}$. 
With iid data, let
\begin{equation}\label{eqn:Omega-bar}
\bar{\matf{\Omega}}=\frac{1}{n}\sum_{i=1}^{n} \vecf{g}_{ni}(\bar{\vecf{\beta}},\tau) \vecf{g}_{ni}(\bar{\vecf{ \beta}},\tau)\tr 
\end{equation}
be an estimator of 
$\matf{\Omega} 
 = \E[ \vecf{g}_{ni}( \vecf{\beta}_{0\tau}, \tau ) 
       \vecf{g}_{ni}( \vecf{\beta}_{0\tau}, \tau )\tr ]$. 
As in (3.11) of \citet{NeweyMcFadden94}, the one-step estimator is 
\begin{equation}\label{eqn:def-est-1s}
\hat{\vecf{\beta}}_{\mathrm{1s}}
= \bar{\vecf{ \beta}} 
  -(\bar{\matf{G}}\tr \bar{\matf{\Omega}}^{-1}\bar{\matf{G}})^{-1}
  \bar{\matf{G}}\tr \bar{\matf{\Omega}}^{-1}
  \sum_{i=1}^{n}\vecf{g}_{ni}(\bar{\vecf{ \beta}},\tau)/n  
. 
\end{equation}

For $\bar{\vecf{\beta}}$, we use \cref{eqn:def-est-MM}. 
For $\bar{ \matf{\Omega}}$ with dependent data, \cref{eqn:Omega-bar} is replaced by a long-run variance estimator as in \citet{NeweyWest87} and \citet{Andrews91}, which we use in our code. 
However, this lacks formal justification; see \cref{sec:asy-normality}. 

\subsection{Two-step {GMM} estimator (over-identification)}
\label{sec:est-GMM}

We also consider the two-step GMM estimator to achieve asymptotic efficiency in over-identified models ($d_Z>d_\beta$).\footnote{This is not always true with time series under fixed-smoothing asymptotics \citep{HwangSun15}.} 
Let $\hat{\matf{W}}$ be a symmetric, positive definite weighting matrix. 
The smoothed GMM estimator minimizes a weighted quadratic norm of the smoothed sample moment vector: 
\begin{align}\label{eqn:def-est-GMM}
\hat{\vecf{\beta}}_{\mathrm{GMM}} & =\argmin_{\vecf{\beta} \in \mathcal{B} } \left[ \sum_{i=1}^{n} \vecf{g}_{ni}(\vecf{\beta}, \tau) \right]\tr \hat{\matf{W} } \left[ \sum_{i=1}^{n} \vecf{g}_{ni}(\vecf{\beta}, \tau) \right]  
 =\argmin_{\vecf{\beta} \in \mathcal{B} } \hat{\vecf{M}}_n(\vecf{\beta},\tau)\tr  \hat{\matf{W} }   \hat{\vecf{M}}_n(\vecf{\beta},\tau).
\end{align}
The usual optimal weighting matrix is an estimator of the inverse long-run variance of the sample moments:\footnote{There are other approaches to achieve efficiency without explicitly estimating the long-run variance, like the (Bayesian) exponentially tilted empirical likelihood of \citet[Thm.\ 3]{Schennach07} and \citet[p.\ 36]{Schennach05}, on which \citet{LancasterJun10} is based.} 
$\hat{\matf{W}}^* = \bar{\matf{\Omega}}^{-1} \pconv \matf{\Omega}^{-1}$, 
where $\bar{\matf{\Omega}}$ depends on an initial estimate $\bar{\vecf{\beta}}$ as in \cref{sec:est-1s}. 
The resulting efficient two-step GMM estimator is 
\begin{equation}\label{eqn:def-est-2s}
\hat{\vecf{\beta}}_{\mathrm{2s}}
= \argmin_{\vecf{\beta} \in \mathcal{B} } 
    \hat{\vecf{M}}_n(\vecf{\beta},\tau)\tr  \bar{\matf{\Omega}}^{-1}
    \hat{\vecf{M}}_n(\vecf{\beta},\tau)
.
\end{equation}

Computing \cref{eqn:def-est-2s} is difficult because the function may be non-convex. 
To find the global minimum, we use the simulated annealing algorithm from the GenSA package in R \citep{R.GenSA}, which is suited to such problems. 
Despite its strengths, simulated annealing cannot reliably solve \cref{eqn:def-est-2s} without an initial value reasonably close to the solution. 
Thankfully, such an initial value is provided by \cref{eqn:def-est-MM} or \cref{eqn:def-est-1s}. 

After computing \cref{eqn:def-est-2s}, one could run simulated annealing again with the \emph{unsmoothed} objective function. 
For linear iid IVQR, \citet{KaplanSun17} suggest that smoothing improves (pointwise in $\tau$) mean squared error but may reduce estimated heterogeneity (across $\tau$), so the benefit of such a final step is ambiguous.

\section{Large sample properties}
\label{sec:asy}

We now establish consistency and asymptotic normality of both the smoothed MM and GMM estimators.

\subsection{Assumptions}
\label{sec:asy-assumptions}

Different subsets of the following assumptions are used for different results. 

\begin{assumption}\label{a:XYZ}
For each observation $i$ among $n$ in the sample, endogenous vector $\vecf{Y}_i\in\mathcal{Y}\subseteq\R^{d_Y}$ and instrument vector $\vecf{Z}_i\in\mathcal{Z}\subseteq\R^{d_Z}$; a subset of $\vecf{Z}_i$ is $\vecf{X}_i\in\mathcal{X}\subseteq\R^{d_X}$, with $d_X\le d_Z$. 
The sequence $\{\vecf{Y}_i,\vecf{Z}_i\}$ is strictly stationary and weakly dependent. 
\end{assumption}
\begin{assumption}\label{a:Lambda}
The function $\Lambda \colon \mathcal{Y}\times\mathcal{X}\times\mathcal{B} \mapsto \R$ is known and has (at least) one continuous derivative in its $\mathcal{B}$ argument for all $\vecf{y}\in\mathcal{Y}$ and $\vecf{x}\in\mathcal{X}$. 
\end{assumption}
\begin{assumption}\label{a:B}
The parameter space $\mathcal{B}\in\R^{d_\beta}$ is compact; $d_\beta \le d_Z$. 
Given $\tau\in(0,1)$, the population parameter $\vecf{\beta}_{0\tau}$ is in the interior of $\mathcal{B}$ and uniquely satisfies the moment condition 
\begin{equation}\label{eqn:ID}
\vecf{0} 
= \E\bigl[ \vecf{Z}_i \bigl( \Ind{\Lambda\bigl(\vecf{Y}_i,\vecf{X}_i,\vecf{\beta}_{0\tau}\bigr) \le 0} - \tau \bigr) \bigr] . 
\end{equation}
\end{assumption}
\begin{assumption}\label{a:Z}
The matrix $\E\left(\vecf{Z}_i\vecf{Z}_i\tr \right)$ is positive definite (and finite). 
\end{assumption}
\begin{assumption}\label{a:Itilde}
The function $\tilde{I}(\cdot)$ satisfies $\tilde{I}(u)=0$ for $u\le-1$, $\tilde{I}(u)=1$ for $u\ge1$, and $-1\le\tilde{I}(u)\le2$ for $-1<u<1$. 
The derivative $\tilde{I}'(\cdot)$ is a symmetric, bounded kernel function of order $r\ge2$, so $\int_{-1}^{1}\tilde{I}'(u)\,du=1$, $\int_{-1}^{1}u^k\tilde{I}'(u)\,du=0$ for $k=1,\ldots,r-1$, and $\int_{-1}^{1}\lvert u^r\tilde{I}'(u) \rvert \,du<\infty$ but $\int_{-1}^{1}u^r\tilde{I}'(u)\,du\ne0$, and $\int_{-1}^{1} \lvert u^{r+1} \tilde{I}'(u) \rvert \,du<\infty$. 
\end{assumption}
\begin{assumption}\label{a:h}
The bandwidth sequence $h_n$ satisfies $h_n=o(n^{-1/(2r)})$. 
\end{assumption}
\begin{assumption}\label{a:Y}
Given any $\vecf{\beta}\in\mathcal{B}$ and almost all $\vecf{Z}_i=\vecf{z}$ (i.e.,\ up to a set of zero probability), 
the conditional distribution of $\Lambda(\vecf{Y}_i,\vecf{X}_i,\vecf{\beta})$ given $\vecf{Z}_i=\vecf{z}$ is continuous in a neighborhood of zero. 
\end{assumption}
\begin{assumption}\label{a:ULLN}
For a fixed $\tau\in(0,1)$, using the definition in \cref{eqn:def-M-hat}, 
\begin{equation}\label{eqn:smooth-Mn-ULLN}
\sup_{\vecf{\beta}\in\mathcal{B}}
\norm{ \hat{\vecf{M}}_n(\vecf{\beta},\tau) - \E[ \hat{\vecf{M}}_n(\vecf{\beta},\tau) ] }
= o_p(1) . 
\end{equation}
\end{assumption}
\begin{assumption}\label{a:U}
Let $\Lambda_i\equiv \Lambda\bigl(\vecf{Y}_i,\vecf{X}_i,\vecf{\beta}_{0\tau}\bigr)$ and $\vecf{D}_i\equiv \nabla_{\vecf{\beta}}\Lambda\bigl(\vecf{Y}_i,\vecf{X}_i,\vecf{\beta}_{0\tau}\bigr)$, 
using the notation
\begin{equation}\label{eqn:def-nabla-b-Lambda}
\vecf{\nabla}_{\vecf{\beta}} \Lambda\left(\vecf{y},\vecf{x},\vecf{\beta}_0\right)
\equiv \pD{}{\vecf{\beta}} \Lambda\left(\vecf{y},\vecf{x},\vecf{\beta}\right) \Bigr\rvert_{\vecf{\beta}=\vecf{\beta}_0} , 
\end{equation}
for the $d_\beta\times1$ partial derivative vector. 
Let $f_{\Lambda|\vecf{Z}}(\cdot\mid\vecf{z})$ denote the conditional PDF of $\Lambda_i$ given $\vecf{Z}_i=\vecf{z}$, and let $f_{\Lambda|\vecf{Z},\vecf{D}}(\cdot\mid\vecf{z},\vecf{d})$ denote the conditional PDF of $\Lambda_i$ given $\vecf{Z}_i=\vecf{z}$ and $\vecf{D}_i=\vecf{d}$. 
(i) For almost all $\vecf{z}$ and $\vecf{d}$, $f_{\Lambda|\vecf{Z}}(\cdot\mid\vecf{z})$ and $f_{\Lambda|\vecf{Z},\vecf{D}}(\cdot\mid\vecf{z},\vecf{d})$ are at least $r$ times continuously differentiable in a neighborhood of zero, where the value of $r$ is from \cref{a:Itilde}. 
For almost all $\vecf{z}\in\mathcal{Z}$ and $u$ in a neighborhood of zero, there exists a dominating function $C(\cdot)$ such that $\abs{f_{\Lambda|\vecf{Z}}^{(r)}(u\mid\vecf{z})}\le C(\vecf{z})$ and $\E\bigl[C(\vecf{Z})\abs{\vecf{Z}}\bigr] < \infty$. 
(ii) The matrix
\begin{equation}\label{eqn:def-G}
\matf{G} 
\equiv 
\left. \pD{}{\vecf{\beta}\tr } \E\bigl[ \vecf{Z}_i \Ind{ \Lambda(\vecf{Y}_i,\vecf{X}_i,\vecf{\beta}) \le 0 } \bigr] \right\rvert_{\vecf{\beta}=\vecf{\beta}_{0\tau}}
= -\E\left\{ \vecf{Z}_i \vecf{D}_i\tr  
          f_{\Lambda|\vecf{Z},\vecf{D}}(0\mid\vecf{Z}_i,\vecf{D}_i)
\right\} 
\end{equation}
has rank $d_\beta$. 
\end{assumption}
\begin{assumption}\label{a:CLT}
A pointwise CLT applies: 
\begin{equation}\label{eqn:CLT}
\sqrt{n} \bigl\{ \hat{\vecf{M}}_n(\vecf{\beta}_{0\tau},\tau) - \E[\hat{\vecf{M}}_n(\vecf{\beta}_{0\tau},\tau)] \bigr\}
\dconv  \Normalp{\vecf{0}}{\matf{\Sigma}_{\tau}} . 
\end{equation}
\end{assumption}
\begin{assumption}\label{a:G-est}
Let $Z_i^{(k)}$ denote the $k$th element of $\vecf{Z}_i$, and similarly $\beta^{(k)}$. 
Let $G_{kj}$ denote the row $k$, column $j$ element of $\matf{G}$ (from \cref{a:U}). 
Assume 
\begin{equation}\label{eqn:G-est}
-\frac{1}{nh_n} \sum_{i=1}^{n} 
 \tilde{I}'\bigl(-\Lambda(\vecf{Y}_i,\vecf{X}_i,\tilde{\vecf{\beta}}_{k})/h_n \bigr)
 Z_i^{(k)} 
 \pD{}{\beta^{(j)}}\Lambda(\vecf{Y}_i,\vecf{X}_i,\vecf{\beta}) 
 \Bigr\rvert_{\vecf{\beta}=\tilde{\vecf{\beta}}_{\tau,k}} 
\pconv  G_{kj} 
 . 
\end{equation}
for each $k=1,\ldots,d_\beta$ and $j=1,\ldots,d_\beta$, where each $\tilde{\vecf{\beta}}_{k}$ lies between $\vecf{\beta}_{0\tau}$ and $\hat{\vecf{\beta}}_{\mathrm{MM}}$ (defined in \cref{a:B} and \cref{eqn:def-est-MM}, respectively). 
\end{assumption}

\begin{assumption}\label{a:W}
For the weighting matrix, $\hat{\matf{W}} \pconv \matf{W}$, and both are symmetric, positive definite matrices. 
\end{assumption}

For transparency, \cref{a:XYZ} includes sampling assumptions that help establish the high-level assumptions \cref{a:ULLN,a:CLT,a:G-est}, which may require additional restrictions on dependence (mixing conditions); see \cref{sec:app-primitive}. 
\Cref{a:Lambda} is stronger than a nonparametric model but more general than a linear-in-parameters model. 
%
\Cref{a:B} assumes global identification, following the GMM tradition going back to \citet[Thm.\ 2.1(iii), p.\ 1035]{Hansen82} due to ``the difficulty of specifying primitive identification conditions for GMM'' \citep[p.\ 2120]{NeweyMcFadden94}, although \citet{ChernozhukovHansen05} have some such results for IVQR. 
\Cref{a:Z} matches Assumption 2(ii) in \citet{KaplanSun17}; it is relatively weak, imposing only a finite second moment on $\vecf{Z}_i$ (and, unlike 2SLS, no moment restrictions on $\vecf{Y}_i$. 
\Cref{a:Itilde} is essentially Assumption 4(i,ii) of \citet{KaplanSun17}. 

\Cref{a:h} ensures that the asymptotic effect of smoothing is negligible; it is relatively weak given $r=4$ (as in \cref{fig:Itilde}), and given the optimal $n^{-1/(2r-1)}$ rate for $h_n$ from \citet{KaplanSun17} for linear IVQR with iid data. 
With weakly dependent data, other bandwidth restrictions are needed to establish \cref{a:G-est}; see \cref{sec:app-primitive-G-est} for details. 
\Cref{sec:app-comp-h} contains suggestions for practical bandwidth selection.

\Cref{a:Y} can be checked easily in most cases. 
For example, if $\Lambda(\vecf{Y}_i,\vecf{X}_i,\vecf{\beta})=\tilde{Y}_i-(D_i,\vecf{X}_i\tr )\vecf{\beta}$ and $\vecf{Z}\tr =(\vecf{X}\tr ,\tilde{Z})$, then \cref{a:Y} is satisfied if the outcome $Y$ has a continuous distribution conditional on almost all $(\vecf{X}\tr ,\tilde{Z})=(\vecf{x}\tr ,\tilde{z})$. 
\Cref{a:ULLN} generally requires some restriction on dependence and moments, but it is much weaker than the iid sampling assumption of \citet[Assumption 2.R1]{ChernozhukovHansen06} or \citet[Assumption 1]{KaplanSun17}. 
\Cref{a:U} is used for the asymptotic normality result; it generalizes parts of Assumptions 3 and 7 in \citet{KaplanSun17} to our nonlinear model. 
The full rank of $\matf{G}$ is also sufficient for local (but not global) identification \citep[e.g.,][p.\ 787]{ChenChernozhukovLeeNewey14}. 
The CLT in \cref{a:CLT} is a high-level assumption, similar to 
condition (iv) in Theorem 7.2 of \citet{NeweyMcFadden94}, for example; like \cref{a:ULLN}, it requires some restriction on dependence and moments but does not require iid sampling. 
Examples of more primitive sufficient conditions for \cref{a:ULLN} and \cref{a:CLT} are given later in 
\cref{sec:app-primitive}. 
\Cref{a:G-est} is actually a generalization of the consistency of Powell's estimator for the asymptotic covariance matrix of the usual quantile regression estimator, as detailed in \cref{sec:asy-normality}. 
\Cref{a:G-est} embodies the stochastic equicontinuity that is often separately assumed, as in Theorem 7.2(v) in \citet{NeweyMcFadden94}; it also involves interrelated restrictions of dependence, moments, and the bandwidth rate, as described in \cref{sec:app-primitive-G-est}. 

\Cref{a:W} is standard for GMM. 
For two-step GMM, it is satisfied given a consistent estimator of the (inverse) asymptotic covariance matrix; see discussion in \cref{sec:asy-normality}. 

Finally, note the lack of a conditional quantile restriction. 
Only the unconditional moments in \cref{a:B} are assumed to be satisfied by the (pseudo) true parameter. 
Thus, all our results hold even under misspecification (of a conditional model).

\subsection{Consistency}
\label{sec:asy-consistency}

To establish consistency, we use Theorem 5.9 in \citet{vanderVaart98}, showing the two required conditions are satisfied here. 
One condition is an identification condition. 
The other requires uniform (in $\vecf{\beta}\in\mathcal{B}$) convergence in probability of the sample maps $\hat{\vecf{M}}_n(\vecf{\beta},\tau)$ to the population map $\vecf{M}(\vecf{\beta},\tau)$. 
No iid sampling assumption is required; the second assumption may be established under weak dependence. 

A detailed example of primitive conditions for the high-level uniform weak law of large numbers assumed in \Cref{a:ULLN} is given in \cref{sec:app-primitive-ULLN}. 

In addition to \cref{a:ULLN}, we must show that the sequence of (non-random) maps $\E[ \hat{\vecf{M}}_n(\vecf{\beta},\tau) ]$ converges to the desired population map $\vecf{M}(\vecf{\beta},\tau)$, as in \cref{lem:smooth-EMn-ULLN}. 

\begin{lemma}\label{lem:smooth-EMn-ULLN}
Under \Cref{a:XYZ,a:Lambda,a:B,a:Z,a:Y,a:Itilde,a:h}, 
for a fixed $\tau$, 
using definitions in \cref{eqn:def-M,eqn:def-M-hat}, 
\begin{equation}\label{eqn:smooth-EMn-ULLN}
\sup_{\vecf{\beta}\in\mathcal{B}}
\absbig{ \E[ \hat{\vecf{M}}_n(\vecf{\beta},\tau) ] 
      - \vecf{M}(\vecf{\beta},\tau) }
= o(1) . 
\end{equation}
\end{lemma}

\Cref{lem:smooth-EMn-ULLN} is intuitive. 
Without smoothing, $\vecf{M}(\cdot)=\E[\hat{\vecf{M}}^u_n(\cdot)]$ for all $n$. 
With smoothing, we should expect this to hold asymptotically if the smoothing is asymptotically negligible. 
The next result establishes consistency.

\begin{theorem}\label{thm:consistency}
Under \Cref{a:XYZ,a:Lambda,a:B,a:Z,a:Y,a:Itilde,a:h,a:ULLN} for smoothed MM, and additionally \cref{a:W} for smoothed GMM, the estimators from \cref{eqn:def-est-MM,eqn:def-est-GMM} are consistent: 
\begin{equation}\label{eqn:consistency}
\hat{\vecf{\beta}}_{\mathrm{MM}} - \vecf{\beta}_{0\tau} = o_p(1) , \quad
\hat{\vecf{\beta}}_{\mathrm{GMM}} - \vecf{\beta}_{0\tau} = o_p(1) . 
\end{equation}
\end{theorem}

\subsection{Asymptotic normality}
\label{sec:asy-normality}

To establish asymptotic normality, smoothing facilitates the usual approach of expanding the sample moments around $\vecf{\beta}_{0\tau}$ because the smoothed sample moments are differentiable. 
That is, we may take a mean value expansion of the first-order condition, rearrange, and take limits.

The following lemma aids the proof of \cref{thm:normality}. 
It relies on \cref{a:CLT} and a proof that the asymptotic ``bias'' is negligible, i.e., 
$\sqrt{n}\E[\hat{\vecf{M}}_n(\vecf{\beta}_{0\tau},\tau)] \to 0$. 

\begin{lemma}\label{lem:Mn0-normality}
Under \Cref{a:XYZ,a:CLT,a:Itilde,a:Z,a:B,a:U,a:h,a:Lambda}, 
\begin{equation}\label{eqn:Mn0-normality}
\sqrt{n}\hat{\vecf{M}}_n(\vecf{\beta}_{0\tau},\tau)
\dconv 
\Normalp{\vecf{0}}{\matf{\Sigma}_{\tau}} , \quad 
\matf{\Sigma}_{\tau} 
= \lim_{n\to\infty} \Varp{ n^{-1/2}\sum_{i=1}^{n} \vecf{g}_{ni}(\vecf{\beta}_{0\tau},\tau) } . 
\end{equation}
With iid data and the conditional quantile restriction $\Pr\bigl( \Lambda(\vecf{Y}_i,\vecf{X}_i,\vecf{\beta}_{0\tau})\le0 \mid \vecf{Z}_i\bigr)=\tau$, then 
$\matf{\Sigma}_{\tau} = \tau(1-\tau) \E(\vecf{Z}_i\vecf{Z}_i\tr)$. 
\end{lemma}

The asymptotic normality of our estimators can now be stated. 
We also show their asymptotically linear (influence function) representations.

\begin{theorem}\label{thm:normality}
Under \Cref{a:XYZ,a:Lambda,a:B,a:Z,a:Y,a:Itilde,a:h,a:ULLN,a:U,a:G-est,a:CLT} for smoothed MM, and additionally \cref{a:W} for smoothed GMM, 
for the estimators from \cref{eqn:def-est-MM,eqn:def-est-GMM}, 
\begin{equation*}
\begin{split}
\sqrt{n} ( \hat{\vecf{\beta}}_{\mathrm{MM}} - \vecf{\beta}_{0\tau} )
&= -\matf{G}^{-1}
    \frac{1}{\sqrt{n}} \sum_{i=1}^{n} 
  \vecf{g}_{ni}(\vecf{\beta}_{0\tau},\tau) 
  +o_p(1) , \\
\sqrt{n} ( \hat{\vecf{\beta}}_{\mathrm{MM}} - \vecf{\beta}_{0\tau} )
&\dconv 
\Normalp{\vecf{0}}{(\matf{G}^{\tr} \matf{\Sigma}_{\tau}^{-1} \matf{G})^{-1}} , \\
\sqrt{n} ( \hat{\vecf{\beta}}_{\mathrm{GMM}} - \vecf{\beta}_{0\tau} )
&= -\{ \matf{G}\tr \matf{W} \matf{G} \}^{-1}
\matf{G}\tr \matf{W}
  \frac{1}{\sqrt{n}} \sum_{i=1}^{n} 
  \vecf{g}_{ni}(\vecf{\beta}_{0\tau},\tau) 
+o_p(1)  , \\
\sqrt{n} ( \hat{\vecf{\beta}}_{\mathrm{GMM}} - \vecf{\beta}_{0\tau} )
&\dconv 
\Normal\bigl( \vecf{0} , (\matf{G}\tr \matf{W} \matf{G} )^{-1} 
\matf{G}\tr \matf{W} \matf{\Sigma}_{\tau} \matf{W} \matf{G}
(\matf{G}\tr \matf{W} \matf{G} )^{-1} \bigr) , 
\end{split}
\end{equation*}
where 
$\matf{G}$ is from \cref{eqn:def-G}, 
$\matf{W}$ is from \cref{a:W}, 
and 
$\matf{\Sigma}_{\tau}$ is from \cref{eqn:Mn0-normality}. 
\end{theorem}

As usual, choosing a weighting matrix such that $\hat{\matf{W}} \pconv \matf{W} = \matf{\Sigma}_{\tau}^{-1}$ is asymptotically efficient in the sense that the resulting asymptotic covariance matrix $(\matf{G}^{\tr} \matf{\Sigma}_{\tau}^{-1} \matf{G})^{-1}$ minus the above GMM covariance matrix is negative semidefinite. 
This is the sense in which the two-step estimator in \cref{eqn:def-est-2s} is efficient.

\Cref{thm:normality} can also be used to construct Wald tests in the usual way. 
The result and its proof are also helpful for constructing ``distance metric'' hypothesis tests and over-identification tests. 
We detail these in a separate paper on inference (in progress).

A consistent long-run variance estimator for quantile models is currently lacking in the literature, as lamented in other quantile papers. 
For example, \citet[p.\ 268]{Kato12} notes that results in \citet{Andrews91} do not apply because they assume smoothness; specifically, Assumptions B(iii) and C(ii) are violated for unsmoothed quantile models. 
Similarly, Assumptions 4 and 5 in \citet{NeweyWest87} are violated, as is Assumption 4 in \citet{deJongDavidson00}. 
Our smoothing yields differentiability, but also a triangular array, violating (2.2) in \citet{Andrews91}, for example. 
However, \citet{deJongDavidson00} allow triangular arrays and generally have very weak conditions. 
In future work, we hope to verify their Assumption 4 for our smoothed quantile GMM setting. 
%

\section{Monte {Carlo} simulations}
\label{sec:sim}

This section reports Monte Carlo simulation results to illustrate the finite-sample performance of the proposed methods. 
Replication code is available on the third author's website.\footnote{It is written in R \citep{R.core} and uses packages from \citet{R.pracma} and \citet{R.GenSA}.}

The following DGPs/models are used; details are in \cref{sec:app-comp-DGPs} and the code. 
DGP 1 has a binary treatment, binary IV, and iid sampling, as with a randomized treatment offer but self-selection into treatment, and the treatment effect increases with the quantile index $\tau$. 
DGPs 2 and 3 are time series regressions with measurement error and either Gaussian (DGP 2) or Cauchy (DGP 3) errors in the outcome equation, but no slope heterogeneity. 
The first three models are exactly identified, so the smoothed MM and GMM estimators are identical; we compare these with the usual QR and IV estimators. 
DGP 4 is for estimating a log-linearized quantile Euler equation using time series data, as in our empirical application; the model is over-identified, so we can compare MM with different GMM estimators.

To quantify precision, instead of root mean squared error (RMSE), we report ``robust RMSE.'' 
This replaces bias with median bias and replaces standard deviation with interquartile range (divided by $1.349$); it equals RMSE if the sampling distribution is normal. 
The primary reason to use the ``robust'' version is that sometimes the usual IV estimator does not even possess a first moment in finite samples, let alone finite variance \citep[e.g.,][]{Kinal80}.%
\footnote{If one really cares about finite-sample RMSE per se, then OLS should be preferred to IV in the cases where the IV RMSE is infinite but the OLS RMSE is finite.} 

\begin{table}[!htb]
\centering
\caption{\label{tab:sim-robust-RMSE1gamma}Simulated precision of estimators of $\gamma_\tau$.}
\sisetup{round-precision=2}
\begin{threeparttable}
\begin{tabular}{ccrcrrrcrrr}
\toprule
&    &       & & \multicolumn{3}{c}{Robust RMSE} & & \multicolumn{3}{c}{Median Bias} \\
\cmidrule{5-7}\cmidrule{9-11}
DGP & $\tau$ & $n$ 
    & & S(G)MM & QR & IV 
    & & S(G)MM & QR & IV \\
\midrule
1                    & $0.25$ & $\num{    20}$ && $\num{  26.51}$ & $\num{  31.81}$ & $\num{  41.71}$ && $\num{  18.00}$ & $\num{  27.87}$ & $\num{  40.30}$ \\
                     & $    $ & $\num{    50}$ && $\num{  19.10}$ & $\num{  27.99}$ & $\num{  40.53}$ && $\num{  11.64}$ & $\num{  26.22}$ & $\num{  39.93}$ \\
                     & $    $ & $\num{   200}$ && $\num{  10.94}$ & $\num{  25.74}$ & $\num{  40.17}$ && $\num{   4.09}$ & $\num{  25.22}$ & $\num{  40.03}$ \\
                     & $    $ & $\num{   500}$ && $\num{   8.61}$ & $\num{  25.33}$ & $\num{  40.12}$ && $\num{   1.54}$ & $\num{  25.07}$ & $\num{  40.07}$ \\
                     & $0.50$ & $\num{    20}$ && $\num{  19.20}$ & $\num{  22.14}$ & $\num{  18.69}$ && $\num{   9.04}$ & $\num{  17.48}$ & $\num{  15.30}$ \\
                     & $    $ & $\num{    50}$ && $\num{  13.87}$ & $\num{  21.16}$ & $\num{  16.47}$ && $\num{   4.92}$ & $\num{  18.86}$ & $\num{  14.93}$ \\
                     & $    $ & $\num{   200}$ && $\num{   8.13}$ & $\num{  20.50}$ & $\num{  15.38}$ && $\num{   1.19}$ & $\num{  19.99}$ & $\num{  15.03}$ \\
                     & $    $ & $\num{   500}$ && $\num{   5.11}$ & $\num{  20.16}$ & $\num{  15.21}$ && $\num{   0.52}$ & $\num{  19.96}$ & $\num{  15.07}$ \\
2                    & $0.25$ & $\num{    20}$ && $\num{   1.34}$ & $\num{   0.59}$ & $\num{   1.34}$ && $\num{  -0.37}$ & $\num{  -0.52}$ & $\num{  -0.24}$ \\
                     & $    $ & $\num{    50}$ && $\num{   0.86}$ & $\num{   0.55}$ & $\num{   0.73}$ && $\num{  -0.09}$ & $\num{  -0.53}$ & $\num{  -0.03}$ \\
                     & $    $ & $\num{   200}$ && $\num{   0.42}$ & $\num{   0.52}$ & $\num{   0.31}$ && $\num{  -0.01}$ & $\num{  -0.51}$ & $\num{   0.02}$ \\
                     & $    $ & $\num{   500}$ && $\num{   0.26}$ & $\num{   0.50}$ & $\num{   0.20}$ && $\num{   0.00}$ & $\num{  -0.50}$ & $\num{   0.02}$ \\
                     & $0.50$ & $\num{    20}$ && $\num{   1.29}$ & $\num{   0.58}$ & $\num{   1.34}$ && $\num{  -0.37}$ & $\num{  -0.51}$ & $\num{  -0.24}$ \\
                     & $    $ & $\num{    50}$ && $\num{   0.83}$ & $\num{   0.54}$ & $\num{   0.73}$ && $\num{  -0.09}$ & $\num{  -0.51}$ & $\num{  -0.03}$ \\
                     & $    $ & $\num{   200}$ && $\num{   0.40}$ & $\num{   0.51}$ & $\num{   0.31}$ && $\num{   0.02}$ & $\num{  -0.50}$ & $\num{   0.02}$ \\
                     & $    $ & $\num{   500}$ && $\num{   0.26}$ & $\num{   0.50}$ & $\num{   0.20}$ && $\num{   0.02}$ & $\num{  -0.50}$ & $\num{   0.02}$ \\
3                    & $0.25$ & $\num{    20}$ && $\num{   2.72}$ & $\num{   0.75}$ & $\num{   3.91}$ && $\num{  -0.42}$ & $\num{  -0.55}$ & $\num{  -0.20}$ \\
                     & $    $ & $\num{    50}$ && $\num{   1.68}$ & $\num{   0.59}$ & $\num{   3.17}$ && $\num{  -0.02}$ & $\num{  -0.51}$ & $\num{   0.19}$ \\
                     & $    $ & $\num{   200}$ && $\num{   0.79}$ & $\num{   0.54}$ & $\num{   2.48}$ && $\num{  -0.02}$ & $\num{  -0.51}$ & $\num{  -0.02}$ \\
                     & $    $ & $\num{   500}$ && $\num{   0.47}$ & $\num{   0.51}$ & $\num{   2.33}$ && $\num{  -0.00}$ & $\num{  -0.50}$ & $\num{   0.06}$ \\
                     & $0.50$ & $\num{    20}$ && $\num{   2.10}$ & $\num{   0.67}$ & $\num{   3.91}$ && $\num{  -0.49}$ & $\num{  -0.53}$ & $\num{  -0.20}$ \\
                     & $    $ & $\num{    50}$ && $\num{   1.36}$ & $\num{   0.56}$ & $\num{   3.17}$ && $\num{  -0.15}$ & $\num{  -0.51}$ & $\num{   0.19}$ \\
                     & $    $ & $\num{   200}$ && $\num{   0.63}$ & $\num{   0.52}$ & $\num{   2.48}$ && $\num{   0.01}$ & $\num{  -0.50}$ & $\num{  -0.02}$ \\
                     & $    $ & $\num{   500}$ && $\num{   0.35}$ & $\num{   0.51}$ & $\num{   2.33}$ && $\num{   0.00}$ & $\num{  -0.50}$ & $\num{   0.06}$ \\
\bottomrule
\end{tabular}
\begin{tablenotes}
  \item $\num{1000}$ replications. ``S(G)MM'' is the estimator in \cref{eqn:def-est-MM,eqn:def-est-GMM} (equivalent here due to exact identification); ``QR'' is quantile regression (no IV); ``IV'' is the usual (mean) IV estimator. 
\end{tablenotes}
\end{threeparttable}
\end{table}

\Cref{tab:sim-robust-RMSE1gamma} shows the precision of our smoothed estimator from \cref{eqn:def-est-MM} (``S(G)MM'') with a very small bandwidth $h=0.0001$ (for simplicity), as well as the usual quantile regression (``QR'') estimator (ignoring endogeneity) and the usual (mean) IV estimator.

\Cref{tab:sim-robust-RMSE1gamma} shows that for all DGPs, the smoothed estimator's robust RMSE declines toward zero as $n$ increases. 
In contrast, the QR estimator's robust RMSE never goes to zero due to endogeneity, and 
the IV estimator's robust RMSE fails to go to zero in DGP 1 where there is heterogeneity across quantiles. 
This reflects the theoretical result that only our estimator is consistent for $\gamma_\tau$ when there is endogeneity and heterogeneity.

\Cref{tab:sim-robust-RMSE1gamma} also shows important finite-sample differences not captured by first-order asymptotics. 
With $n=20$, for DGP 2 or 3, the lowest robust RMSE is actually that of QR: despite its (median) bias being the largest due to ignoring the endogeneity, its dispersion is so much smaller than the other estimators' dispersions that its overall robust RMSE is the smallest. 
This advantage persists to $n=50$, but eventually $n$ is large enough for the (median) bias to dominate. 
In DGPs 2 and 3 that lack slope heterogeneity, the IV estimator is the most efficient when errors are Gaussian (and $n$ is large enough), but not with Cauchy errors, reflecting the greater efficiency of the median (over the mean) when errors are heavy-tailed.

\Cref{tab:sim-GMM} compares simulated robust RMSE of three smoothed estimators (all with $h=0.1$) of log-linearized quantile Euler equations, specifically the EIS parameter. 
Compared to the GMM estimator with identity weighting matrix, the two-step GMM estimator is always more efficient. 
The biggest such advantage is with the smallest sample size, $n=50$; this seems surprising since the two-step GMM's estimated weighting matrix has the largest variance in that case. 
Two-step GMM is not always better (or always worse) than the MM estimator that takes the linear projection of the (lone) endogenous regressor onto the vector of (five) instruments to be the second excluded instrument (in addition to the constant). 
Two-step GMM has a smaller robust RMSE in some cases, even half that of MM with $\tau=0.25$ and $n=500$, but in other cases MM has smaller robust RMSE, especially when $n=50$. 
Perhaps 
the additional variance of the two-step estimator due to its use of a long-run variance estimator (for the weighting matrix) makes it less efficient in these cases, 
a phenomenon explored in non-quantile GMM by \citet{HwangSun15}. 
Alternatively, perhaps future work can improve the long-run variance estimator's precision, in turn improving the two-step estimator's precision.

\begin{table}[htbp]
\centering
\caption{\label{tab:sim-GMM}Simulated precision of smoothed estimators of EIS.}
\sisetup{round-precision=3,round-mode=places,table-format=-1.3}
\begin{threeparttable}
\begin{tabular}{ccrcS[table-format=1.3]S[table-format=1.3]S[table-format=1.3]cSSS}
\toprule
      &    &       & & \multicolumn{3}{c}{Robust RMSE} & & \multicolumn{3}{c}{Median Bias} \\
\cmidrule{5-7}\cmidrule{9-11}
&&&&&\multicolumn{2}{c}{GMM} &&&\multicolumn{2}{c}{GMM}\\
\cmidrule{6-7}\cmidrule{10-11}
DGP & $\tau$ & $n$ 
& & {MM} & {(2s)} & {(ID)}
& & {MM} & {(2s)} & {(ID)} \\
\midrule
4 & 0.25 & \numnornd{    50} &&  0.17999837062 &  0.28503968098 &  0.31483778489 &&  0.05704911633 &  0.05478882235 &  0.05385485166 \\
4 & 0.25 & \numnornd{   200} &&  0.12936146919 &  0.09718086708 &  0.10711493621 &&  0.02367156155 &  0.01281140224 &  0.01725805701 \\
4 & 0.25 & \numnornd{   500} &&  0.09702102954 &  0.04682820187 &  0.04687576707 &&  0.00774964285 &  0.00346598677 &  0.00252535822 \\
4 & 0.50 & \numnornd{    50} &&  0.15257915384 &  0.25812098531 &  0.30035712298 &&  0.05542071769 &  0.02281211013 &  0.02461849143 \\
4 & 0.50 & \numnornd{   200} &&  0.09887603827 &  0.12430833031 &  0.14381088407 &&  0.02505958952 &  0.00373343422 & -0.00817069504 \\
4 & 0.50 & \numnornd{   500} &&  0.06584933962 &  0.07926093079 &  0.08900325642 &&  0.01396784182 & -0.00333198778 & -0.00732897789 \\
\bottomrule
\end{tabular}
\begin{tablenotes}
\item $\numnornd{1000}$ replications. ``MM'' is the estimator in \cref{eqn:def-est-MM}; ``GMM(2s)'' is the estimator in \cref{eqn:def-est-2s}; ``GMM(ID)'' is the estimator in \cref{eqn:def-est-GMM} with identity weighting matrix. 
\end{tablenotes}
\end{threeparttable}
\end{table}

\section{Application: quantile {Euler} equation}
\label{sec:EIS}

This section illustrates the usefulness of the proposed methods through an empirical example: the estimation of a quantile Euler equation. 
We apply the proposed methodology to an economic model of intertemporal allocation of consumption and estimate the elasticity of intertemporal substitution (EIS). 
The EIS is a parameter of central importance in macro\-economics and finance. 
We refer to \citet{Campbell03}, \citet{Cochrane05}, and \citet{LjungqvistSargent12}, and the references therein, for a comprehensive overview.

There is a large empirical literature that attempts to estimate the EIS; among others, \citet{HansenSingleton83}, \citet{Hall88}, \citet{CampbellMankiw89}, \citet{CampbellViceira99}, \citet{Campbell03}, and \citet{Yogo04}. The majority of the literature relies on the traditional expected utility framework.  
The purpose of this application is to estimate and make inference on the EIS for selected developed countries in \citeposs{Campbell03} data set using the quantile utility maximization model. The quantile model has useful advantages, such as robustness, ability to capture heterogeneity, and separation the notion of risk attitude from the intertemporal substitution.

\Cref{sec:EIS-econ model} describes in detail the model that leads to the quantile Euler equation, establishing parallels with the standard expected utility Euler equation. 
\Cref{sec:EIS-est} describes the estimation procedure, and \cref{sec:EIS-log-linearization} discusses log-linearization for the quantile model. 
\Cref{sec:EIS-interpretation} discusses an interpretation of the parameters in question. 
In \cref{sec:EIS-data} we review the data, and finally \cref{sec:EIS-results} presents the empirical results.

\subsection{Description of the economic model}
\label{sec:EIS-econ model}

\Citet{deCastroGalvao17} employ a variation of the standard economy model of \citet{Lucas78}. 
The economic agents decide on the intertemporal consumption and savings (assets to hold) over an infinity horizon economy, subject to a linear budget constraint. 
The decision generates an intertemporal policy function, which is used to estimate the parameters of interest for a given utility function. 
Their work is related to that of \citet{Giovannetti13}, who works with a similar model but restricts the analysis to two periods, whereas \citet{deCastroGalvao17} consider an infinite horizon.

The specific model is as follows. 
Let $C_{t}$ denote the amount of consumption good that the individual consumes in period $t$. 
At the beginning of period $t$, the consumer has $x_{t}$ units of the risky asset, which pays dividend $d_{t}$. 
The price of the consumption good is normalized to one, while the price of the risky asset in period $t$ is $p(d_{t})$. 
Then, the consumer decides its consumption $C_{t}$ and how many units of the risky asset $x_{t+1}$ to save for the next period, subject to the budget constraint
\begin{align}\label{eq:budget constr}
C_{t} + p(d_{t}) x_{t+1} &\le \left[ d_{t} + p(d_{t})  \right] x_{t},
\end{align}
and positivity restriction
\begin{align}\label{eq:nonneg constr}
C_{t}, x_{t+1} &\ge 0. 
\end{align}
In equilibrium, we have that $x_{t}^{\ast}=1, \forall t,k$. 

So far, the model is exactly the same as the standard Lucas' model, but the objective function will differ. 
In the standard model, the consumer maximizes 
\begin{equation}\label{eq:ee maxim}
\E\left[ \sum_{t=0}^{\infty} \beta^{t} U (C_{t}) \;\biggm\vert\; \Omega_{0} \right], 
\end{equation}
subject to \cref{eq:budget constr,eq:nonneg constr}, where $\beta \in (0,1)$ is the discount factor, $U \colon \R_{+} \mapsto \R$ is the utility function, and $\Omega_0$ is the information set at time $t=0$. 
For the expected utility choice problem, dynamic consistency and the principle of optimality for \cref{eq:ee maxim} imply that at time $s \ge 1$, the consumer chooses $\{C_{t},x_{t}\}_{t \ge s}$ to maximize 
\begin{equation}\label{eq:ee maxim s}
\E \left[ \sum_{t=s}^{\infty}  \beta^{t-s} U (C_{t}) \;\biggm\vert\; \Omega_{s} \right], 
\end{equation}
subject to \cref{eq:budget constr,eq:nonneg constr}.  
The connection between problems \cref{eq:ee maxim,eq:ee maxim s} is made explicit by the linearity of the expectation operator and the law of iterated expectations: 
\begin{align}\notag
&\E \left[ \sum_{t=0}^{\infty} \beta^{t} U (C_{t}) \;\biggm\vert\; \Omega_{0} \right]
\\&=
 U (C_{0}) 
 + \E\biggl[ \beta U(C_{1}) 
          +\E\Bigl[ \beta^{2} U(C_{2}) 
                   +  \E\bigl[ \beta^{3} U(C_{3}) + \cdots \mid \Omega_{2} \bigr]
             \,\bigm\vert\, \Omega_{1} \Bigr] 
    \,\Bigm\vert\,  \Omega_{0} \biggr] . 
\label{eq:eu law it exp}
\end{align}

\Citet{deCastroGalvao17} replace the operator $\E$ in \eqref{eq:eu law it exp} with $\Q_{\tau}$, such that the consumer maximizes the following quantile objective function
\begin{align}
\label{eq:quant expansion} & 
  U (C_{0}) 
 + \Q_\tau\biggl[ \beta_{\tau} U(C_{1}) 
          +\Q_\tau\Bigl[ \beta_{\tau}^{2} U(C_{2}) 
                   +  \Q_\tau\bigl[ \beta_{\tau}^{3} U(C_{3}) + \cdots \mid \Omega_{2} \bigr]
             \,\bigm\vert\, \Omega_{1} \Bigr] 
    \,\Bigm\vert\,  \Omega_{0} \biggr]\\
\notag & = \Q_\tau\biggl[\Q_\tau\Bigl[\Q_\tau\bigl[ U (C_{0}) 
 + \beta_{\tau} U(C_{1}) + \beta_{\tau}^{2} U(C_{2}) 
                   + \beta_{\tau}^{3} U(C_{3}) + \cdots \mid \Omega_{2} \bigr]
             \,\bigm\vert\, \Omega_{1} \Bigr] 
    \,\Bigm\vert\,  \Omega_{0} \biggr]\\
\notag & \equiv \Q_{\tau}^{\infty} \left[\sum_{t=0}^{\infty} \beta_{\tau}^{t} U(C_{t}) \right],
\end{align}
again subject to \cref{eq:budget constr,eq:nonneg constr}, where $\beta_{\tau}\in(0,1)$ is the discount factor for the quantile $\tau$. 


Unfortunately, linearity and the law of iterated expectations do not hold for the $\tau$-quantile operator, $\Q_{\tau}$. 
Thus, in order to preserve dynamic consistency and the principle of optimality, we need to maintain the structure developed in \cref{eq:quant expansion}. 
\Citet{deCastroGalvao17} show that the limit above exists and is well defined. 
Moreover, they show that the quantile preferences are dynamically consistent, the principle of optimality holds, and the corresponding dynamic problem yields a value function, via a fixed-point argument. They further provide conditions so that the value function is differentiable and concave.

When comparing the expected and quantile utility models, we note that the structure in the right-hand side of \cref{eq:eu law it exp} reflects the following associated value function, 
\begin{equation}\label{eqn:EU-value-fn}
v(x_{t},d_{t}) = \max_{ x_{t+1} \ge 0 } \left\{ U\bigl( [d_{t}+p(d_{t})]x_{t} - p(d_{t})x_{t+1} \bigr) + \beta  \E[  v(x_{t+1},d_{t+1}) \mid \Omega_{t} ] \right\}. 
\end{equation}
The value function for the quantile problem is the same as \cref{eqn:EU-value-fn} but with $\Q_\tau$ replacing $\E$: 
\begin{equation}\label{eqn:QU-value-fn}
v(x_{t},d_{t}) = \max_{ x_{t+1} \ge 0 } \left\{ U\bigl( [d_{t}+p(d_{t})]x_{t} - p(d_{t})x_{t+1} ) + \beta_{\tau}  \Q_{\tau}[  v(x_{t+1},d_{t+1}) \mid \Omega_{t} ] \right\}. 
\end{equation}

In addition, \citet{deCastroGalvao17} derive the corresponding Euler equation, using the fact that  in equilibrium, the holdings are $x_{t}=1$ for all $t$: 
\begin{equation}\label{eq:Euler applied to prob}
- p(d_{t}) U'(C_{t}) + \beta_{\tau} \Q_{\tau}[ U'(C_{t+1})  
(d_{t+1} + p(d_{t+1}) ) \mid \Omega_{t} ] 
= 0 .
\end{equation}
Defining the asset's return by 
\begin{equation*}\label{eq:def retorno}
1+r_{t+1} \equiv \frac{ d_{t+1} + p(d_{t+1})}{p(d_{t})} , 
\end{equation*}
the Euler equation in \cref{eq:Euler applied to prob} simplifies to 
\begin{equation}
\label{eq:aee}
 \Q_{\tau}\left [\beta_{\tau} (1+r_{t+1})  \frac{U'(C_{t+1})}{U' (C_{t})} \,\biggm\vert\,  \Omega_{t} \right] = 1.
\end{equation} 
After parameterizing the utility function, \cref{eq:aee} is a conditional quantile restriction in the form of \cref{eq:qrmodel}, as in our econometric model. 

The quantile Euler equation in \cref{eq:aee} looks similar to the standard Euler equation from expected utility maximization, 
\begin{equation}\label{eq:aeee}
\E\left [\beta (1+r_{t+1})  \frac{U'(C_{t+1})}{U' (C_{t})} \,\biggm\vert\, \Omega_{t} \right] = 1.
\end{equation} 
The expressions inside the conditional quantile and conditional expectation are identical.

For obtaining the mentioned results,  \citet{deCastroGalvao17}  assume the following. 
\begin{assumption}\label{a:pricing}
 \begin{enumerate}
 \item\label{a:pricing-interval} The dividends assume values in $\mathcal{Z} \subseteq \R$, which  is a bounded interval, and  $\mathcal{X} =[0,\bar{x}]$ for some $\bar{x}>1$;
 \item\label{a:pricing-Markov}  $\{d_{t}\}$ is a Markov process with PDF   $f \colon \mathcal{Z} \times \mathcal{Z} \mapsto \R_{+}$, which is continuous, 
 symmetric ($f(a,b)=f(b,a)$), 
 $f(d_{t},d_{t+1})>0$ for all $(d_{t},d_{t+1}) \in \mathcal{Z} \times \mathcal{Z}$, and satisfies the property that if $h \colon \mathcal{Z} \mapsto \R$ is weakly increasing and  $z \le z'$, then 
 \begin{equation}\label{eq:Mono K}
 \int_{\mathcal{Z}} h(\alpha) f(\alpha \mid z) \, d \alpha  \le 
  \int_{\mathcal{Z}} h(\alpha) f(\alpha \mid z') \, d \alpha ;
 \end{equation}
\item\label{a:pricing-utility} $U \colon \R_{+} \mapsto \R$ is given by $U(c)= \frac{1}{1-\gamma}c^{1-\gamma}$, for $\gamma>0$;
\item\label{a:pricing-z} $z \mapsto z+p(z)$ is $C^{1}$ and non-decreasing, with $\D{}{z} z\left[\ln (z+p(z)) \right] \ge \gamma$. 
\end{enumerate} 
\end{assumption}

Assumptions \cref{a:pricing}\cref{a:pricing-interval,a:pricing-Markov,a:pricing-utility} are standard in economic applications. 
In \Cref{a:pricing}\cref{a:pricing-z}, it is natural to expect that the price $p(z)$ is non-decreasing with the dividend $z$, and $z+p(z)$ being non-decreasing is an even weaker requirement.

\subsection{Estimation}
\label{sec:EIS-est}

We follow a large body of the literature \citep[e.g.,][]{Campbell03} and use isoelastic utility, 
\begin{equation}\label{eqn:utility-fn}
U(C_{t})=\frac{1}{1-\gamma}C_{t}^{1-\gamma}, \quad \gamma>0 . 
\end{equation}
The ratio of marginal utilities is 
\begin{equation}\label{eq:crra}
 \frac{U'(C_{t+1})}{U' (C_{t})} = \left( \frac{C_{t+1}}{C_{t}} \right)^{-\gamma}.
\end{equation}
From \cref{eq:aee,eq:crra}, the Euler equation is thus 
\begin{equation}\label{eq:enlqr}
\Q_{\tau}\left[ \beta_\tau (1+r_{t+1}) ( C_{t+1}/C_{t} )^{-\gamma_\tau} -1 \mid \Omega_{t}\right] = 0. 
\end{equation}
The quantile Euler equation in \cref{eq:enlqr} is a conditional quantile restriction with finite-dimensional parameter vector $(\beta_\tau,\gamma_\tau)$, as in \cref{eq:qrmodel}, so we may use smoothed GMM estimation.

\subsection{Log-linearization}
\label{sec:EIS-log-linearization}

One benefit of the quantile Euler equation is that it may be log-linearized with no approximation error, unlike the standard Euler equation. 
One may rewrite \cref{eq:enlqr} as 
\begin{equation}\label{eqn:epsilon}
\Q_\tau[\epsilon_{t+1}\mid\Omega_t]=1 , \quad 
\epsilon_{t+1} \equiv \beta (1+r_{t+1}) ( C_{t+1} / C_{t} )^{-\gamma} . 
\end{equation}
For general random variable $W$, $\Q_\tau[\ln(W)]=\ln(\Q_\tau[W])$ (``equivariance'') since $\ln(\cdot)$ is strictly increasing and continuous. 
In contrast, $\E[\ln(W)] \le \ln(\E[W])$ by Jensen's inequality. 
Continuing from \cref{eqn:epsilon}, 
\begin{align}\notag
\ln(\epsilon_{t+1})
  &= \ln(\beta) + \ln(1+r_{t+1}) - \gamma \ln(C_{t+1}/C_t) , \\
\ln(C_{t+1}/C_t) 
  &= \gamma^{-1} \ln(\beta) 
    +\gamma^{-1} \ln(1+r_{t+1}) 
    -\gamma^{-1} \ln(\epsilon_{t+1}) 
. \label{eq:loglin} 
\end{align}
If $\gamma>0$, then 
since $\Q_\tau(W)=-\Q_{1-\tau}(-W)$ (and $0=-0$), 
\begin{align*}
0 &= \ln(1) 
  = \Q_\tau[\ln(\epsilon_{t+1}) \mid\Omega_t ]
  = \Q_{1-\tau}[-\gamma^{-1}\ln(\epsilon_{t+1}) \mid \Omega_t ] 
\\&= \Q_{1-\tau}[\ln(C_{t+1}/C_t) - \gamma^{-1} \ln(\beta) 
                - \gamma^{-1} \ln(1+r_{t+1}) \mid \Omega_t ] 
. 
\end{align*}
Thus, 
$\ln(\beta)/\gamma$ and $1/\gamma$ are the intercept and slope (respectively) of the $1-\tau$ IV quantile regression of $\ln(C_{t+1}/C_t)$ on a constant and $\ln(1+r_{t+1})$, with instruments from $\Omega_t$.

Similarly, in the sample, the $g_{ni}$ should be equivalent for nonlinear and log-linear estimation since $\hat{\epsilon}_{t+1} \le 1  \iff  \ln(\hat{\epsilon}_{t+1}) \le 0$. 
%
Thus, the corresponding estimators should be identical. 
In our application, this is generally true (matching 2+ significant figures), although sometimes there are differences due to the numerical methods, especially simulated annealing, since the log-linear minimization is done in a transformed parameter space. 
Additionally, the nonlinear and log-linear estimators do not match when the latter is negative (implying misspecification) since the above arguments assume $\gamma>0$. 
%

\subsection{Interpretation}
\label{sec:EIS-interpretation}

The parameters of interest in \eqref{eq:enlqr} are $\beta_{\tau}$ and $\gamma_{\tau}$. 
The former is the usual discount factor. 
The parameter $1/\gamma_{\tau}$ is the standard measure of EIS implicit in the CRRA utility function in \eqref{eqn:utility-fn}. 
The EIS is a measure of responsiveness of the consumption growth rate to the real interest rate. 
As in \citet{Hall88}, in a model with uncertainty, the interpretation is similar, and a high value of EIS means that when the real interest rate is expected to be high, the consumer will actively defer consumption to the later period.

The interpretation of $1/\gamma_{\tau}$ as the EIS remains valid for the quantile maximization model.\footnote{\citeposs{Hall88} argument that $\gamma$ fundamentally represents the EIS rather than 
risk aversion applies here, too.} 
Most directly, this can be seen in equation \eqref{eq:loglin}, where $1/\gamma$ is the derivative of $\ln(C_{t+1}/C_{t})$ with respect to $\ln(1+r_{t+1})$, holding $\epsilon_{t+1}$ constant.

\subsection{Data}
\label{sec:EIS-data}

We use data originally from \cite{Campbell03} and provided by \citet{Yogo04}.\footnote{\url{https://sites.google.com/site/motohiroyogo/research/EIS_Data.zip}} 
It consists of aggregate level quarterly data for the United States (US), United Kingdom (UK), Australia (AUS), and Sweden (SWE). 
The sample period for the US is 1947Q3--1998Q4, 
UK is 1970Q3--1999Q1, 
Australia is 1970Q3--1998Q4, 
and 
Sweden is 1970Q3--1999Q2. 
Consumption is measured at the beginning of the period, consisting of nondurables plus services for the US and total consumption for the other countries, in real, per capita terms. 
The real interest rate deflates a proxy for the nominal short-term rate by the consumer price index. 
Instruments are lags of log real consumption growth, nominal interest rate, inflation, and a log dividend-price ratio for equities. 
For a complete description of the data, see \citet{Campbell03}.

\subsection{Results}
\label{sec:EIS-results}

Other than using quantiles, our estimation follows Table 2 of \citet[p.\ 805]{Yogo04}. 
\Citet{Yogo04} uses 2SLS to estimate the (structural) log-linearized model $\ln(C_{t+1}/C_t) = \delta_0 + \delta_1 \ln(1+r_{t+1}) + u_{t+1}$, where $r_{t+1}$ is the real interest rate, instrumenting for $\ln(1+r_{t+1})$ with twice lagged measures of nominal interest rate, inflation, consumption growth, and log dividend-price ratio, where $\delta_1=1/\gamma$ is the EIS and $\delta_0=\ln(\beta)/\gamma$. 
\Citet{Yogo04} emphasizes that these are strong instruments that predict the real interest rate well, although formally characterizing ``strong'' for IVQR remains an open question.%
\footnote{In contrast, when trying to estimate the EIS by 2SLS of $\ln(1+r_{t+1})$ on $\ln(C_{t+1}/C_t)$, or replacing the real interest rate with a real stock index return, the instruments are weak because it is difficult to predict consumption growth or stock returns.}

\Cref{tab:EIS-Yogo-beta,tab:EIS-Yogo-gamma} show the quantile Euler equation estimates for $\beta_\tau$ and $\gamma_\tau$ (respectively), 
using the results from \cref{sec:EIS-log-linearization}, with the smoothed MM estimator in \cref{eqn:def-est-MM} and the smoothed two-step GMM estimator in \cref{eqn:def-est-2s}, for the deciles $\tau=0.1,\ldots,0.9$. 
For two-step GMM, the long-run variance estimator follows \citet{Andrews91} with a quadratic spectral kernel. 
For both estimators, the plug-in bandwidth from \citet{KaplanSun17} was used. 
For comparison, 2SLS estimates are in each table's bottom row.

\sisetup{round-precision=2,round-mode=places,table-format=1.2}
\begin{table}[!htb]
\centering
\caption{\label{tab:EIS-Yogo-beta}Smoothed MM and GMM estimates of $\beta_\tau$, log-linear model.}
\begin{threeparttable}
\begin{tabular}[c]{C BBB BBB BBB BBB}
\toprule
 & &  \multicolumn{2}{c}{US} && \multicolumn{2}{c}{UK} && \multicolumn{2}{c}{AUS} && \multicolumn{2}{c}{SWE} \\
\cmidrule{3-4}\cmidrule{6-7}\cmidrule{9-10}\cmidrule{12-13}
{$\tau$} && {$\hat{\beta}_\mathrm{MM}$} & {$\hat{\beta}_\mathrm{GMM}$} && {$\hat{\beta}_\mathrm{MM}$} & {$\hat{\beta}_\mathrm{GMM}$} && {$\hat{\beta}_\mathrm{MM}$} & {$\hat{\beta}_\mathrm{GMM}$} && {$\hat{\beta}_\mathrm{MM}$} & {$\hat{\beta}_\mathrm{GMM}$}\\
\midrule
0.1 &&      0.917 &      0.918  &&      0.000 &      0.118  &&      0.883 &      0.906  &&      0.948 &      0.944 \\
0.2 &&      0.896 &      0.924  &&      1.860 &      0.753  &&      0.875 &      0.873  &&      0.954 &      0.951 \\
0.3 &&      0.853 &      1.133  &&      1.164 &      0.823  &&      0.818 &      0.749  &&      0.954 &      0.951 \\
0.4 &&      0.109 &      1.043  &&      1.072 &      0.867  &&      2.375 &      1.353  &&      0.952 &      0.947 \\
0.5 &&      1.144 &      1.022  &&      1.040 &      1.032  &&      1.125 &      1.092  &&      0.899 &      0.803 \\
0.6 &&      1.041 &      1.013  &&      1.016 &      1.011  &&      1.020 &      1.028  &&      0.941 &      1.067 \\
\rowstyle{\bfseries}
0.7 &&      1.015 &      1.008  &&      0.998 &      1.001  &&      0.998 &      1.001  &&      0.970 &      0.971 \\
\rowstyle{\bfseries}
0.8 &&      1.003 &      1.002  &&      0.981 &      0.990  &&      0.980 &      0.988  &&      0.970 &      0.970 \\
0.9 &&      0.990 &      0.995  &&      0.958 &      0.973  &&      0.959 &      0.976  &&      0.965 &      0.963 \\
\midrule
{2SLS} && \multicolumn{2}{c}{      1.08} && \multicolumn{2}{c}{      1.03} && \multicolumn{2}{c}{      1.11} && \multicolumn{2}{c}{      0.27}\\
\bottomrule
\end{tabular}
%
\end{threeparttable}
\end{table}

\sisetup{round-precision=1,round-mode=places,table-format=-2.1}
\begin{table}[!htb]
\centering
\caption{\label{tab:EIS-Yogo-gamma}Smoothed MM and GMM estimates of $\gamma_\tau$, log-linear model.}
\begin{threeparttable}
\begin{tabular}[c]{C BBB BBB BBB BBB}
\toprule
 & &  \multicolumn{2}{c}{US} && \multicolumn{2}{c}{UK} && \multicolumn{2}{c}{AUS} && \multicolumn{2}{c}{SWE} \\
\cmidrule{3-4}\cmidrule{6-7}\cmidrule{9-10}\cmidrule{12-13}
{$\tau$} && {$\hat{\gamma}_\mathrm{MM}$} & {$\hat{\gamma}_\mathrm{GMM}$} && {$\hat{\gamma}_\mathrm{MM}$} & {$\hat{\gamma}_\mathrm{GMM}$} && {$\hat{\gamma}_\mathrm{MM}$} & {$\hat{\gamma}_\mathrm{GMM}$} && {$\hat{\gamma}_\mathrm{MM}$} & {$\hat{\gamma}_\mathrm{GMM}$}\\
\midrule
0.1 &&     -7.232 &     -7.099  &&   -744.010 &   -114.148  &&     -6.946 &     -5.454  &&     -3.368 &     -3.740 \\
0.2 &&    -11.507 &     -8.021  &&     43.113 &    -18.958  &&     -8.962 &     -9.125  &&     -4.001 &     -4.434 \\
0.3 &&    -19.967 &     16.521  &&     13.460 &    -15.667  &&    -17.035 &    -24.638  &&     -5.465 &     -6.113 \\
0.4 &&   -342.539 &      6.902  &&      8.268 &    -14.110  &&     98.933 &     33.935  &&     -8.669 &    -10.203 \\
0.5 &&     26.491 &      4.531  &&      7.347 &      5.342  &&     20.651 &     14.765  &&    -38.725 &    -89.528 \\
0.6 &&     10.879 &      3.659  &&      6.574 &      4.054  &&      7.251 &      8.011  &&   -150.986 &    301.480 \\
\rowstyle{\bfseries}
0.7 &&      6.587 &      3.424  &&      5.859 &      3.691  &&      5.902 &      5.107  &&     13.946 &     13.576 \\
\rowstyle{\bfseries}
0.8 &&      4.944 &      3.240  &&      5.327 &      3.685  &&      5.039 &      3.719  &&      5.330 &      6.445 \\
0.9 &&      4.871 &      3.086  &&      5.020 &      3.970  &&      4.416 &      0.000  &&      3.328 &      4.821 \\
\midrule
{2SLS} && \multicolumn{2}{c}{      16.7} && \multicolumn{2}{c}{       6.0} && \multicolumn{2}{c}{      22.1} && \multicolumn{2}{c}{    -544.4}\\
\bottomrule
\end{tabular}
%
\end{threeparttable}
\end{table}

\Cref{tab:EIS-Yogo-beta,tab:EIS-Yogo-gamma} generally show economically unrealistic estimates at lower $\tau$ but plausible estimates at larger $\tau$. 
For $\tau \le 0.4$, some of the $\beta_\tau$ estimates are unrealistically far from one, and many of the $\gamma_\tau$ estimates are negative. 
For $\tau \ge 0.5$, in contrast, most of the $\beta_\tau$ estimates are close to one, and most $\gamma_\tau$ estimates seem plausible. 
For $\tau \in \{0.7,0.8\}$ in particular, looking across all four countries and both MM and GMM estimates, the estimates are all contained within the ranges $\hat\beta_\tau \in [0.97, 1.02]$ and $\hat\gamma_\tau \in [3.2, 13.9]$.

The differences between MM and two-step GMM estimates are often relatively small, especially when the estimates are reasonable. 
However, the table shows some economically significant differences, such as for the US (even with $\tau \ge 0.6$).

The differences between the quantile and 2SLS estimates can be economically significant. 
This includes the case of Sweden, where the 2SLS estimates are entirely unrealistic: $\hat\beta_\mathrm{2SLS}=0.27$ and $\hat\gamma_\mathrm{2SLS}=-544.4$. 
Although smaller $\tau$ produce unrealistic estimates, the Sweden quantile estimates for $\tau=0.7$ and $\tau=0.8$ have $\hat\beta_\tau=0.97$ and $\hat\gamma_\tau$ in the range $[5.3, 13.9]$, all perfectly reasonable.%
\footnote{Further, among the other seven countries whose data \citet{Yogo04} examined (Netherlands, Canada, France, Germany, Italy, Japan, Switzerland), all seven had negative 2SLS estimates $\hat\gamma_\mathrm{2SLS}<0$, but five of the seven had positive $\hat\gamma_\tau>0$ with $\tau=0.9$ (and plausible $\hat\beta_\tau$).} 
For the other countries, the $\tau=0.5$ estimates are most similar to 2SLS, but $\tau \ge 0.7$ leads to more realistic $\hat{\beta}_\tau$ and smaller $\hat{\gamma}_\tau$.

In all, this empirical application illustrates that the quantile utility maximization model and new smoothed estimators serve as important tools to study economic behavior.

\section{Conclusion}
\label{sec:conclusion}

For finite-dimensional parameters defined by general quantile-type restrictions, we have developed smoothed MM and GMM estimation and asymptotic theory, for exactly and over-identified models, respectively, allowing for weakly dependent data and nonlinear models. 
This includes nonlinear IV quantile regression and quantile Euler equations as special cases, and our theory is robust to misspecification of the structural models.

The empirical results suggest that quantile utility maximization combined with our smoothed estimation can provide a useful, economically meaningful alternative to estimation based on expected utility. 
A bonus feature is the ability to log-linearize the quantile Euler equation without any approximation error, unlike the standard Euler equation. 
Future work may apply our methods to household panel data or carefully consider how to determine $\tau$.

There is more to explore econometrically, too: 
quantile GMM inference (in progress), 
IVQR averaging estimators (in progress), 
optimal bandwidth choice, 
non/semiparametric models, 
fixed-smoothing asymptotic approximations, 
higher-order bootstrap refinements, 
formally establishing \cref{a:G-est} when $\vecf{D}_i$ depends on $\vecf{\beta}_{0\tau}$, 
and results uniform in $\tau$, 
among other topics. 

\appendix

\section{Proofs}
\label{sec:app-proofs}

\begin{proof}[Proof of \Cref{prop:local-ID}]
See Appendix A.1 of \citet{ChenChernozhukovLeeNewey14}. 
\end{proof}

\begin{proof}[Proof of \Cref{lem:smooth-EMn-ULLN}]
Noting that $\absbig{\vecf{Z}[\tilde{I}(\cdot)-\Ind{\cdot}]} \le \abs{\vecf{Z}}$ (i.e., $\abs{\vecf{Z}}$ is a dominating function) and applying the dominated convergence theorem (since $\vecf{Z}$ has finite expectation by \cref{a:Z}), since $h_n\to0$ by \cref{a:h}, 
\begin{align*}
& \lim_{h_n\to0} \sup_{\vecf{\beta}\in\mathcal{B}} 
\normbig{  \E[ \hat{\vecf{M}}_n(\vecf{\beta},\tau) ] 
       -\E\left[\vecf{Z}\left(\Ind{\Lambda(\vecf{Y},\vecf{X},\vecf{\beta})\le0}-\tau\right) \right]
}
\\&= \lim_{h_n\to0} \max_{\vecf{\beta}\in\mathcal{B}} 
\normbig{ \E\left\{\vecf{Z}\left[\tilde{I}\left(\frac{-\Lambda(\vecf{Y},\vecf{X},\vecf{\beta})}{h_n}\right) - \Ind{\Lambda(\vecf{Y},\vecf{X},\vecf{\beta})\le0}\right] \right\}
}
\\&= \lim_{h_n\to0} 
\normbig{ \E\left\{\vecf{Z}\left[\tilde{I}\left(\frac{-\Lambda(\vecf{Y},\vecf{X},\vecf{\beta}^*_n)}{h_n}\right) - \Ind{\Lambda(\vecf{Y},\vecf{X},\vecf{\beta}^*_n)\le0}\right] \right\}
}
\\&= \normbig{ \E\left\{
      \lim_{h_n\to0} \vecf{Z} \left[\tilde{I}\left(\frac{-\Lambda(\vecf{Y},\vecf{X},\vecf{\beta}^*_n)}{h_n}\right) - \Ind{\Lambda(\vecf{Y},\vecf{X},\vecf{\beta}^*_n)\le0}\right] 
   \right\}
}
\\&= \vecf{0} 
\label{eqn:smooth-unsmooth-limit-DCT}\refstepcounter{equation}\tag{\theequation}
\end{align*}
as long as there is no probability mass at $\Lambda(\vecf{Y},\vecf{X},\vecf{\beta})=0$ for any $\vecf{\beta}\in\mathcal{B}$ and almost all $\vecf{Z}$, which is indeed true by \Cref{a:Y}. 
The notation $\vecf{\beta}^*_n$ denotes the value attaining the maximum, which exists since $\mathcal{B}$ is compact by \cref{a:B}. 
\end{proof}

\begin{proof}[Proof of \Cref{thm:consistency}]

We first prove consistency of the smoothed method of moments estimator. 
We then prove consistency of the smoothed GMM estimator. 

\paragraph{MM Consistency.} To prove consistency of $\hat{\vecf{\beta}}_{\mathrm{MM}}$, we show that the conditions of Theorem 5.9 in \citet{vanderVaart98} are satisfied. 
Alternatively, one could apply Theorem 2.1 in \citet[p.\ 2121]{NeweyMcFadden94}, where $\hat{\vecf{\beta}}$ maximizes $\hat{Q}_n(\vecf{\beta}) \equiv -\lVert\hat{\vecf{M}}_n(\vecf{\beta},\tau)\rVert$ with $\hat{Q}_n(\hat{\vecf{\beta}})=0$. 

Combining results from \cref{a:ULLN,lem:smooth-EMn-ULLN} and the triangle inequality,
\begin{align*}
& \sup_{\vecf{\beta}\in\mathcal{B}}
\absbig{ \hat{\vecf{M}}_n(\vecf{\beta},\tau) - 
\vecf{M}(\vecf{\beta},\tau)
}
\\&=
\sup_{\vecf{\beta}\in\mathcal{B}}
\absbig{ \hat{\vecf{M}}_n(\vecf{\beta},\tau) 
    - \E\bigl[ \hat{\vecf{M}}_n(\vecf{\beta},\tau) \bigr] 
    + \E\bigl[ \hat{\vecf{M}}_n(\vecf{\beta},\tau) \bigr] 
    - \vecf{M}(\vecf{\beta},\tau) }
\\&\le
\overbrace{\sup_{\vecf{\beta}\in\mathcal{B}}
\absbig{ \hat{\vecf{M}}_n(\vecf{\beta},\tau) 
    - \E\bigl[ \hat{\vecf{M}}_n(\vecf{\beta},\tau) \bigr] }}^{=o_p(1)\textrm{ by \cref{a:ULLN}}}
+
\overbrace{\sup_{\vecf{\beta}\in\mathcal{B}}
\absbig{ \E\bigl[ \hat{\vecf{M}}_n(\vecf{\beta},\tau) \bigr] 
    - \vecf{M}(\vecf{\beta},\tau) }}^{=o_p(1)\textrm{ by \cref{lem:smooth-EMn-ULLN}}}
\\&= o_p(1)+o_p(1)
   = o_p(1) . 
\label{eqn:Mn-pt-Thm59-cond1}\refstepcounter{equation}\tag{\theequation}
\end{align*}
This satisfies the first condition of Theorem 5.9 in \citet[p.\ 46]{vanderVaart98}, or (combined with the continuity of $\normbig{\cdot}$) condition (iv) in Theorem 2.1 of \citet{NeweyMcFadden94}. 

For the second condition of Theorem 5.9 in \citet{vanderVaart98}, since $\mathcal{B}$ is a compact subset of Euclidean space, 
so is the set 
\begin{equation*}
\{\vecf{\beta} : \norm{ \vecf{\beta}-\vecf{\beta}_{0\tau} } \ge \epsilon, 
                 \vecf{\beta} \in \mathcal{B} \} 
\end{equation*}
for any $\epsilon>0$. 
Writing out 
\begin{align*}
\vecf{M}(\vecf{\beta},\tau)
&= \E\left\{ \vecf{Z}_i \left[ \Ind{\Lambda(\vecf{Y}_i,\vecf{X}_i,\vecf{\beta})\le0} - \tau \right] \right\} 
\\&= \E\left( \E\left\{ \vecf{Z}_i \left[ \Ind{\Lambda(\vecf{Y}_i,\vecf{X}_i,\vecf{\beta})\le0} - \tau \right] \mid \vecf{Z}_i \right\} \right)
\\&= \E\left\{ \vecf{Z}_i \left[ \Pr\bigl( \Lambda(\vecf{Y}_i,\vecf{X}_i,\vecf{\beta}) \le 0 \mid \vecf{Z}_i \bigr) - \tau \right] \right\} , 
\end{align*}
we see that the function $\vecf{M}(\vecf{\beta},\tau)$ is continuous in $\vecf{\beta}$ given \cref{a:Y,a:Lambda}. 
Note that \cref{a:Lambda} alone is not sufficient: it implies $\lim_{\vecf{\delta}\to\vecf{0}}\Lambda(\vecf{Y}_i,\vecf{X}_i,\vecf{\beta}+\vecf{\delta})\to \Lambda(\vecf{Y}_i,\vecf{X}_i,\vecf{\beta})$ (for any realization $\omega\in\Omega$ in the implicit underlying probability space), but $\Ind{\cdot\le0}$ is not a continuous function. 
Specifically, it is discontinuous at zero, so the continuous mapping theorem only guarantees convergence (for $\omega\in\Omega$) where $\Lambda(\vecf{Y}_i,\vecf{X}_i,\vecf{\beta}) \ne 0$. 
\Cref{a:Y} assumes this is a zero probability event (conditional on almost all $\vecf{Z}_i$), so $\Ind{\Lambda(\vecf{Y}_i,\vecf{X}_i,\vecf{\beta}+\vecf{\delta}) \le 0}$ still converges almost surely to $\Ind{\Lambda(\vecf{Y}_i,\vecf{X}_i,\vecf{\beta}) \le 0}$ as $\vecf{\delta}\to\vecf{0}$ (i.e., the set of $\omega\in\Omega$ for which it does not converge has measure zero). 
Altogether, by \cref{a:Lambda,a:Y}, the bounded convergence theorem, and the continuous mapping theorem, writing $\Lambda_i\equiv \Lambda(\vecf{Y}_i,\vecf{X}_i,\vecf{\beta})$, 
\begin{align*}
& \lim_{\vecf{\delta}\to\vecf{0}}
\Pr\bigl( \Lambda(\vecf{Y}_i,\vecf{X}_i,\vecf{\beta}+\vecf{\delta}) \le 0 \mid \vecf{Z}_i \bigr)
\\&= \lim_{\vecf{\delta}\to\vecf{0}}
\E\bigl( \Ind{\Lambda(\vecf{Y}_i,\vecf{X}_i,\vecf{\beta}+\vecf{\delta}) \le 0} \mid \vecf{Z}_i \bigr)
\\&= \E\bigl( \lim_{\vecf{\delta}\to\vecf{0}} \Ind{\Lambda(\vecf{Y}_i,\vecf{X}_i,\vecf{\beta}+\vecf{\delta}) \le 0} \mid \vecf{Z}_i \bigr)
\\&= \E\bigl( \Ind{\Lambda(\vecf{Y}_i,\vecf{X}_i,\vecf{\beta}) \le 0} \mid \vecf{Z}_i , \Lambda_i\ne0 \bigr)
     \overbrace{\Pr(\Lambda_i\ne0\mid\vecf{Z}_i)}^{=1\textrm{ a.s., by \cref{a:Y}}}
\\&\quad+ 
  \E\bigl( \lim_{\vecf{\delta}\to\vecf{0}} \Ind{\Lambda(\vecf{Y}_i,\vecf{X}_i,\vecf{\beta}+\vecf{\delta}) \le 0} \mid \vecf{Z}_i , \Lambda_i=0 \bigr)
  \overbrace{\Pr( \Lambda_i=0 \mid \vecf{Z}_i )}^{=0\textrm{ a.s., by \cref{a:Y}}} 
\\&= \E\bigl( \Ind{\Lambda(\vecf{Y}_i,\vecf{X}_i,\vecf{\beta}) \le 0} \mid \vecf{Z}_i \bigr) 
\end{align*}
almost surely. 

Since a continuous function on a compact set attains a minimum, letting $\vecf{\beta}^*$ denote the minimizer, 
\begin{equation}\label{eqn:Mn-pt-Thm59-cond2}
\inf_{\vecf{\beta}:\norm{\vecf{\beta}-\vecf{\beta}_{0\tau}}\ge\epsilon}
\norm{ \vecf{M}(\vecf{\beta},\tau) }
=
\min_{\vecf{\beta}:\norm{\vecf{\beta}-\vecf{\beta}_{0\tau}}\ge\epsilon}
\norm{ \vecf{M}(\vecf{\beta},\tau) }
=
\norm{ \vecf{M}(\vecf{\beta}^*,\tau) }
> 0
\end{equation}
by \cref{a:B}, which says that for any $\vecf{\beta}\ne\vecf{\beta}_{0\tau}$, $\vecf{M}(\vecf{\beta},\tau)\ne\vecf{0}$, so $\norm{\vecf{M}(\vecf{\beta}^*,\tau)}>0$ (since $\norm{\cdot}$ is a norm). 
Alternatively, for the conditions in Theorem 2.1 of \citet{NeweyMcFadden94}, 
(i) and (ii) are directly assumed in our \cref{a:B}, and (iii) is satisfied by the continuity of $\vecf{M}(\cdot,\tau)$ (as shown above). 

Consistency of $\hat{\vecf{\beta}}_{\mathrm{MM}}$ follows by Theorem 5.9 in \citet{vanderVaart98} or Theorem 2.1 in \citet{NeweyMcFadden94}.

\paragraph{GMM Consistency.} To prove the consistency of $\hat{\vecf{\beta}}_{\mathrm{GMM}}$, we show that the two conditions of Theorem 5.7 in \citet{vanderVaart98} are satisfied. 
The first condition of Theorem 5.7 in \citet{vanderVaart98} requires 
\begin{equation}\label{eqn:Thm5.7-cond1}
\sup_{\beta \in \mathcal{B}} 
\abs{ \hat{\vecf{M}}_n\left(\vecf{\beta},\tau\right)\tr 
\hat{\matf{W} }   \hat{\vecf{M}}_n\left(\vecf{\beta},\tau\right) -\vecf{M}\left(\vecf{\beta},\tau\right)\tr  
\matf{W}
\vecf{M}\left(\vecf{\beta},\tau\right) }
\pconv 0.
\end{equation}
From \cref{eqn:Mn-pt-Thm59-cond1}, 
$\sup_{\vecf{\beta}\in\mathcal{B}}
 \norm{ \hat{\vecf{M}}_n(\vecf{\beta},\tau) - 
\vecf{M}(\vecf{\beta},\tau)
}  = o_p(1)$. 
From \cref{a:W}, $ \hat{\matf{W}} = \matf{W} + o_p(1) $, which does not depend on $\vecf{\beta}$. 

Let $\norm{\cdot}$ denote the Frobenius matrix norm 
$\norm{\matf{A}}=\norm{\matf{A}\tr}=\sqrt{\textrm{tr}(\matf{A}\matf{A}\tr)}$, 
which is the Euclidean norm if $\matf{A}$ is a vector. 
Given this norm, the Cauchy--Schwarz inequality states that for any matrices $\matf{A}$ and $\matf{B}$,  $\norm{\matf{A}\matf{B} } \le \norm{\matf{A}} \norm{\matf{B}} $. 

We now use the triangle inequality, Cauchy--Schwarz inequality, uniform convergence in probability of $\hat{\vecf{M}}_n(\vecf{\beta},\tau)$, and convergence in probability of $\hat{\matf{W}}$, to show the required condition in \cref{eqn:Thm5.7-cond1}: 
\begin{align*}
& \sup_{\beta \in \mathcal{B}} 
\bigl\lvert
 \hat{\vecf{M}}_n(\vecf{\beta},\tau)\tr  \hat{\matf{W} }   
      \hat{\vecf{M}}_n(\vecf{\beta},\tau) 
     -\vecf{M}(\vecf{\beta},\tau)\tr  \matf{W}   
      \vecf{M}(\vecf{\beta},\tau) 
\bigr\rvert 
\\&= \sup_{\beta \in \mathcal{B}} 
\bigl\lvert
(  \vecf{M}(\vecf{\beta},\tau)
 +(\hat{\vecf{M}}(\vecf{\beta},\tau)-\vecf{M}(\vecf{\beta},\tau)  ))\tr 
(\matf{W} + (\hat{\matf{W}}-\matf{W})) 
\\&\qquad\quad\times
(  \vecf{M}(\vecf{\beta},\tau)
 +(\hat{\vecf{M}}(\vecf{\beta},\tau) - \vecf{M}(\vecf{\beta},\tau)  )) 
\\&\qquad\quad
-\vecf{M}(\vecf{\beta},\tau)\tr  \matf{W}  \vecf{M}(\vecf{\beta},\tau) 
\bigr\rvert 
\\&= \sup_{\beta \in \mathcal{B}}  
\bigl\lvert
   \vecf{M}(\vecf{\beta},\tau)\tr 
   (\hat{\matf{W}}-\matf{W})    
   \vecf{M}(\vecf{\beta},\tau)
+
  \vecf{M}(\vecf{\beta},\tau)\tr 
  \matf{W} 
  (\hat{\vecf{M}}(\vecf{\beta},\tau)-\vecf{M}(\vecf{\beta},\tau)  )
\\&\quad 
+ 
  (\hat{\vecf{M}}(\vecf{\beta},\tau)-\vecf{M}(\vecf{\beta},\tau) )\tr
  \matf{W}
  \vecf{M}(\vecf{\beta},\tau) 
+ 
  \vecf{M}(\vecf{\beta},\tau)\tr 
  (\hat{\matf{W}}-\matf{W}) 
  (\hat{\vecf{M}}(\vecf{\beta},\tau)-\vecf{M}(\vecf{\beta},\tau) )
\\&\quad + 
  (\hat{\vecf{M}}(\vecf{\beta},\tau)-\vecf{M}(\vecf{\beta},\tau) )\tr
  \matf{W} 
  (\hat{\vecf{M}}(\vecf{\beta},\tau)-\vecf{M}(\vecf{\beta},\tau) )
\\&\quad + 
   (\hat{\vecf{M}}(\vecf{\beta},\tau)-\vecf{M}(\vecf{\beta},\tau) )\tr
   (\hat{\matf{W}}-\matf{W})    
   \vecf{M}(\vecf{\beta},\tau)
\\
&\quad + 
   (\hat{\vecf{M}}(\vecf{\beta},\tau)-\vecf{M}(\vecf{\beta},\tau) )\tr
   (\hat{\matf{W}}-\matf{W})    
   (\hat{\vecf{M}}(\vecf{\beta},\tau)-\vecf{M}(\vecf{\beta},\tau) )
\bigr\rvert \\
&\le \sup_{\beta \in \mathcal{B}}  
\abs{ 
   \vecf{M}(\vecf{\beta},\tau)\tr 
   (\hat{\matf{W}}-\matf{W})    
   \vecf{M}(\vecf{\beta},\tau)
} 
+ \sup_{\beta \in \mathcal{B}} 
\abs{
  \vecf{M}(\vecf{\beta},\tau)\tr 
  \matf{W} 
  (\hat{\vecf{M}}(\vecf{\beta},\tau)-\vecf{M}(\vecf{\beta},\tau)  )
} \\
&\quad + \sup_{\beta \in \mathcal{B}} 
\abs{
  (\hat{\vecf{M}}(\vecf{\beta},\tau)-\vecf{M}(\vecf{\beta},\tau) )\tr
  \matf{W}
  \vecf{M}(\vecf{\beta},\tau) 
} 
+ \sup_{\beta \in \mathcal{B}} 
\abs{ 
  \vecf{M}(\vecf{\beta},\tau)\tr 
  (\hat{\matf{W}}-\matf{W}) 
  (\hat{\vecf{M}}(\vecf{\beta},\tau)-\vecf{M}(\vecf{\beta},\tau) )
}\\
&\quad + \sup_{\beta \in \mathcal{B}} 
\abs{
  (\hat{\vecf{M}}(\vecf{\beta},\tau)-\vecf{M}(\vecf{\beta},\tau) )\tr
  \matf{W} 
  (\hat{\vecf{M}}(\vecf{\beta},\tau)-\vecf{M}(\vecf{\beta},\tau) )
}\\
&\quad + \sup_{\beta \in \mathcal{B}} 
\abs{
   (\hat{\vecf{M}}(\vecf{\beta},\tau)-\vecf{M}(\vecf{\beta},\tau) )\tr
   (\hat{\matf{W}}-\matf{W})    
   \vecf{M}(\vecf{\beta},\tau)
}\\
&\quad + \sup_{\beta \in \mathcal{B}} 
\abs{
   (\hat{\vecf{M}}(\vecf{\beta},\tau)-\vecf{M}(\vecf{\beta},\tau) )\tr
   (\hat{\matf{W}}-\matf{W})    
   (\hat{\vecf{M}}(\vecf{\beta},\tau)-\vecf{M}(\vecf{\beta},\tau) )
   }\\
&= \sup_{\beta \in \mathcal{B}}  
\norm{ 
   \vecf{M}(\vecf{\beta},\tau)\tr 
   (\hat{\matf{W}}-\matf{W})    
   \vecf{M}(\vecf{\beta},\tau)
} 
+ \sup_{\beta \in \mathcal{B}} 
\norm{
  \vecf{M}(\vecf{\beta},\tau)\tr 
  \matf{W} 
  (\hat{\vecf{M}}(\vecf{\beta},\tau)-\vecf{M}(\vecf{\beta},\tau)  )
} \\
&\quad + \sup_{\beta \in \mathcal{B}} 
\norm{
  (\hat{\vecf{M}}(\vecf{\beta},\tau)-\vecf{M}(\vecf{\beta},\tau) )\tr
  \matf{W}
  \vecf{M}(\vecf{\beta},\tau) 
} 
\\&\quad + \sup_{\beta \in \mathcal{B}} 
\norm{ 
  \vecf{M}(\vecf{\beta},\tau)\tr 
  (\hat{\matf{W}}-\matf{W}) 
  (\hat{\vecf{M}}(\vecf{\beta},\tau)-\vecf{M}(\vecf{\beta},\tau) )
}\\
&\quad + \sup_{\beta \in \mathcal{B}} 
\norm{
  (\hat{\vecf{M}}(\vecf{\beta},\tau)-\vecf{M}(\vecf{\beta},\tau) )\tr
  \matf{W} 
  (\hat{\vecf{M}}(\vecf{\beta},\tau)-\vecf{M}(\vecf{\beta},\tau) )
}\\
&\quad + \sup_{\beta \in \mathcal{B}} 
\norm{
   (\hat{\vecf{M}}(\vecf{\beta},\tau)-\vecf{M}(\vecf{\beta},\tau) )\tr
   (\hat{\matf{W}}-\matf{W})    
   \vecf{M}(\vecf{\beta},\tau)
}\\
&\quad + \sup_{\beta \in \mathcal{B}} 
\norm{
   (\hat{\vecf{M}}(\vecf{\beta},\tau)-\vecf{M}(\vecf{\beta},\tau) )\tr
   (\hat{\matf{W}}-\matf{W})    
   (\hat{\vecf{M}}(\vecf{\beta},\tau)-\vecf{M}(\vecf{\beta},\tau) )
   }\\
&\le \sup_{\beta \in \mathcal{B}}  
\overbrace{
        \norm{   \vecf{M}(\vecf{\beta},\tau)\tr }
\norm{
   (\hat{\matf{W}}-\matf{W})  }
\norm{
   \vecf{M}(\vecf{\beta},\tau)
} }^{\textrm{by Cauchy--Schwarz inequality}}
+ \sup_{\beta \in \mathcal{B}} 
\norm{
  \vecf{M}(\vecf{\beta},\tau)\tr }
\norm{ \matf{W} }
\norm{ (\hat{\vecf{M}}(\vecf{\beta},\tau)-\vecf{M}(\vecf{\beta},\tau)  )
} \\
&\quad + \sup_{\beta \in \mathcal{B}} 
\norm{
  (\hat{\vecf{M}}(\vecf{\beta},\tau)-\vecf{M}(\vecf{\beta},\tau) )\tr }
\norm{  \matf{W} }
\norm{  \vecf{M}(\vecf{\beta},\tau) 
} \\
&\quad + \sup_{\beta \in \mathcal{B}} 
\norm{ 
  \vecf{M}(\vecf{\beta},\tau)\tr }
\norm{  (\hat{\matf{W}}-\matf{W}) }
\norm{  (\hat{\vecf{M}}(\vecf{\beta},\tau)-\vecf{M}(\vecf{\beta},\tau) )
}\\
&\quad + \sup_{\beta \in \mathcal{B}} 
\norm{
  (\hat{\vecf{M}}(\vecf{\beta},\tau)-\vecf{M}(\vecf{\beta},\tau) )\tr}
\norm{  \matf{W} }
\norm{  (\hat{\vecf{M}}(\vecf{\beta},\tau)-\vecf{M}(\vecf{\beta},\tau) )
}\\
&\quad + \sup_{\beta \in \mathcal{B}} 
\norm{
   (\hat{\vecf{M}}(\vecf{\beta},\tau)-\vecf{M}(\vecf{\beta},\tau) )\tr}
\norm{   (\hat{\matf{W}}-\matf{W})    }
\norm{  \vecf{M}(\vecf{\beta},\tau)
}\\
&\quad + \sup_{\beta \in \mathcal{B}} 
\norm{
   (\hat{\vecf{M}}(\vecf{\beta},\tau)-\vecf{M}(\vecf{\beta},\tau) )\tr}
\norm{   (\hat{\matf{W}}-\matf{W})  }  
\norm{   (\hat{\vecf{M}}(\vecf{\beta},\tau)-\vecf{M}(\vecf{\beta},\tau) )
   }\\
&\le \sup_{\beta \in \mathcal{B}}  
 \norm{ \vecf{M}(\vecf{\beta},\tau)\tr }
   \norm{ (\hat{\matf{W}}-\matf{W})  }
\overbrace{\sup_{\beta \in \mathcal{B}} 
   \norm{ \vecf{M}(\vecf{\beta},\tau)
} }^{=O(1)}
\\&\quad + \sup_{\beta \in \mathcal{B}} 
    \norm{ \vecf{M}(\vecf{\beta},\tau)\tr }
  \norm{ \matf{W} }
\sup_{\beta \in \mathcal{B}} 
  \norm{ (\hat{\vecf{M}}(\vecf{\beta},\tau)-\vecf{M}(\vecf{\beta},\tau)  )
} 
\\&\quad + \sup_{\beta \in \mathcal{B}} 
\norm{
  (\hat{\vecf{M}}(\vecf{\beta},\tau)-\vecf{M}(\vecf{\beta},\tau) )\tr }
 \overbrace{ \norm{  \matf{W} } }^{=O(1)}
\sup_{\beta \in \mathcal{B}}
  \norm{  \vecf{M}(\vecf{\beta},\tau) 
} \\
&\quad + \sup_{\beta \in \mathcal{B}} 
\norm{ 
  \vecf{M}(\vecf{\beta},\tau)\tr }
   \norm{  (\hat{\matf{W}}-\matf{W}) }
\sup_{\beta \in \mathcal{B}} 
   \norm{  (\hat{\vecf{M}}(\vecf{\beta},\tau)-\vecf{M}(\vecf{\beta},\tau) )
}\\
&\quad + \sup_{\beta \in \mathcal{B}} 
\norm{
  (\hat{\vecf{M}}(\vecf{\beta},\tau)-\vecf{M}(\vecf{\beta},\tau) )\tr}
   \norm{  \matf{W} }
\sup_{\beta \in \mathcal{B}} 
   \norm{  (\hat{\vecf{M}}(\vecf{\beta},\tau)-\vecf{M}(\vecf{\beta},\tau) )
}\\
&\quad + \sup_{\beta \in \mathcal{B}} 
\norm{
   (\hat{\vecf{M}}(\vecf{\beta},\tau)-\vecf{M}(\vecf{\beta},\tau) )\tr}
\norm{   (\hat{\matf{W}}-\matf{W})    }
\sup_{\beta \in \mathcal{B}} 
   \norm{  \vecf{M}(\vecf{\beta},\tau)
}\\
&\quad + \sup_{\beta \in \mathcal{B}} 
\norm{
   (\hat{\vecf{M}}(\vecf{\beta},\tau)-\vecf{M}(\vecf{\beta},\tau) )\tr}
\norm{   (\hat{\matf{W}}-\matf{W})  }  
\sup_{\beta \in \mathcal{B}} 
  \norm{   (\hat{\vecf{M}}(\vecf{\beta},\tau)-\vecf{M}(\vecf{\beta},\tau) )
   }\\
&=o_p(1)+o_p(1)+o_p(1)+o_p(1)+o_p(1)+o_p(1)+o_p(1)
 =o_p(1).
\end{align*} 
Above, we know $\norm{\matf{W}}=O(1)$ since $\matf{W}$ is fixed, 
and we know $\sup_{\vecf{\beta} \in \mathcal{B}} 
\norm{\vecf{M}(\vecf{\beta},\tau) } = O(1)$ 
since $\vecf{M}(\vecf{\beta},\tau) $ is continuous in $\vecf{\beta}$ and $ \mathcal{B}$ is a compact set.

The second condition of Theorem 5.7 in \citet{vanderVaart98} is that $\vecf{\beta}_{0\tau} $ satisfies the well-separated minimum property. 
Since $\vecf{M}(\vecf{\beta},\tau)\tr  \matf{W}    \vecf{M}(\vecf{\beta},\tau) $ is continuous in $\vecf{\beta}$ and 
$\{ \vecf{\beta} : \norm{\vecf{\beta}-\vecf{\beta}_{0\tau}}\ge\epsilon, \vecf{\beta}\in\mathcal{B} \}$ 
is a compact set, let $\vecf{\beta}^{*}$ denote the minimizer: for any $\epsilon>0 $,
\begin{equation}\label{eqn:Mn-pt-Thm57-cond2}
\inf_{\vecf{\beta}:\norm{\vecf{\beta}-\vecf{\beta}_{0\tau}}\ge\epsilon}
\vecf{M}(\vecf{\beta},\tau)\tr  \matf{W}    \vecf{M}(\vecf{\beta},\tau)
= \vecf{M}(\vecf{\beta}^{*},\tau)\tr  \matf{W }   \vecf{M}(\vecf{\beta}^{*},\tau) . 
\end{equation}
By \cref{a:B}, $\vecf{M}(\vecf{\beta},\tau)\ne\vecf{0}$ for any $\vecf{\beta}\ne\vecf{\beta}_{0\tau}$, 
so $\vecf{M}(\vecf{\beta}^*,\tau)\ne\vecf{0}$. 
Since $\matf{W}$ is positive definite (\cref{a:W}), 
\begin{equation*} 
\vecf{M}(\vecf{\beta}^{*},\tau)\tr \matf{ W  }  \vecf{M}(\vecf{\beta}^{*},\tau) 
> 0.
\end{equation*}
Thus, for any $\epsilon>0$, 
\begin{equation}\label{eqn:Thm5.7-cond2-GMM}
\inf_{\vecf{\beta}:\norm{\vecf{\beta}-\vecf{\beta}_{0\tau}}\ge\epsilon}
\vecf{M}(\vecf{\beta},\tau)\tr  \matf{W }   \vecf{M}(\vecf{\beta},\tau)
> 0
= \vecf{M}(\vecf{\beta}_{0 \tau},\tau)\tr \matf{ W}    \vecf{M}(\vecf{\beta}_{0\tau},\tau) . 
\end{equation} 
Consistency of $\hat{\vecf{\beta}}_{\mathrm{GMM}}$ follows by Theorem 5.7 in \citet{vanderVaart98}.
\end{proof}

\begin{proof}[Proof of \Cref{lem:Mn0-normality}]
Decomposing into a mean-zero term and a ``bias'' term, 
\begin{align*}
\sqrt{n}\hat{\vecf{M}}_n(\vecf{\beta}_{0\tau},\tau)
= \overbrace{\sqrt{n} \bigl\{ \hat{\vecf{M}}_n(\vecf{\beta}_{0\tau},\tau) 
                  -\E[ \hat{\vecf{M}}_n(\vecf{\beta}_{0\tau},\tau) ] \bigr\}
            }^{\dconv \Normalp{\vecf{0}}{\matf{\Sigma}_{\tau}}\textrm{ by \cref{a:CLT}}}
+ \overbrace{\sqrt{n} \E[ \hat{\vecf{M}}_n(\vecf{\beta}_{0\tau},\tau) ]}^{\textrm{want to show }o_p(1)} . 
\end{align*}
With iid data, \citet[Thm.\ 1]{KaplanSun17} show $\matf{\Sigma}_{\tau}=\tau(1-\tau) \E( \vecf{Z}_i\vecf{Z}_i\tr  )$. 
The remainder of the proof shows that the second term is indeed $o_p(1)$, actually $o(1)$. 

Let $\Lambda_i\equiv \Lambda(\vecf{Y}_i,\vecf{X}_i,\vecf{\beta}_{0\tau})$, with marginal PDF $f_\Lambda(\cdot)$ and conditional PDF $f_{\Lambda|\vecf{Z}}(\cdot\mid \vecf{z})$ given $\vecf{Z}_i=\vecf{z}$. 
Given strict stationarity of the data, using the definitions in \cref{eqn:def-M-hat}, assuming the support of $\Lambda_i$ given $\vecf{Z}_i=\vecf{z}$ is the interval $[\Lambda_L(\vecf{z}),\Lambda_H(\vecf{z})]$ with $\Lambda_L(\vecf{z})\le-h_n\le h_n\le \Lambda_H(\vecf{z})$, 
\begin{align*}
& \E[ \hat{\vecf{M}}_n(\vecf{\beta}_{0\tau},\tau) ]
= \E\left[ \frac{1}{n} \sum_{i=1}^{n} \vecf{g}_n(\vecf{Y}_i,\vecf{X}_i,\vecf{Z}_i,\vecf{\beta}_{0\tau},\tau) \right]
= \E\left[ \vecf{g}_n(\vecf{Y}_i,\vecf{X}_i,\vecf{Z}_i,\vecf{\beta}_{0\tau},\tau) \right] 
\\&= \E\left\{ \vecf{Z}_i [ \tilde{I}(-\Lambda_i/h_n) - \tau ] \right\} 
\\&= \E\left\{ \vecf{Z}_i \E[ \tilde{I}(-\Lambda_i/h_n) - \tau \mid \vecf{Z}_i ] \right\} 
\\&= \E\biggl\{ \vecf{Z}_i \overbrace{
               \int_{\Lambda_L(\vecf{Z}_i)}^{\Lambda_H(\vecf{Z}_i)} [\tilde{I}(-L/h_n) - \tau]
               \,dF_{\Lambda|Z}(L\mid\vecf{Z}_i)}^{\textrm{integrate by parts}} \biggr\} 
\\&= \E\biggl\{ \vecf{Z}_i \biggl[
       \overbrace{\left. \left( \tilde{I}(-L/h_n)-\tau \right) F_{\Lambda|Z}(L\mid \vecf{Z}_i) 
       \right\rvert_{\Lambda_L(\vecf{Z}_i)}^{\Lambda_H(\vecf{Z}_i)}}^{=-\tau\textrm{: use \cref{a:Itilde} and $\Lambda_H(\vecf{Z}_i)\ge h_n$}}
\\&\qquad\qquad\qquad
        -\int_{\Lambda_L(\vecf{Z}_i)}^{\Lambda_H(\vecf{Z}_i)}
         F_{\Lambda|Z}(L\mid \vecf{Z}_i) \overbrace{\tilde{I}'(-L/h_n)}^{=0\textrm{ for }L\not\in[-h_n,h_n]} (-h_n^{-1} )
         \,dL 
        \biggr] \biggr\} 
\\&= \E\biggl\{ \vecf{Z}_i \biggl[ -\tau
       +\overbrace{h_n^{-1} \int_{-h_n}^{h_n}
         F_{\Lambda|Z}(L\mid \vecf{Z}_i) \tilde{I}'(-L/h_n) 
         \,dL }^{\textrm{change of variables to }v=-L/h_n} 
        \biggr] \biggr\} 
\\&= \E\left\{ \vecf{Z}_i \left[ -\tau
       +\int_{-1}^{1} F_{\Lambda|Z}(-h_n v\mid \vecf{Z}_i) \tilde{I}'(v) 
         \,dv \right] \right\} 
\\&= \E\left\{ \vecf{Z}_i \left[ -\tau
       +\int_{-1}^{1} \left( \sum_{k=0}^{r}F_{\Lambda|Z}^{(k)}(0\mid \vecf{Z}_i)\frac{(-h_n)^k v^k}{k!} \right) \tilde{I}'(v) 
         \,dv \right] \right\} 
\\&\quad+ \E\biggl\{ \vecf{Z}_i 
       \int_{-1}^{1} 
        \overbrace{f_{\Lambda|Z}^{(r)}(-\tilde{h}v \mid \vecf{Z}_i)}^{\tilde{h}\in[0,h_n]\textrm{ (from MVT)}} 
        \frac{(-h_n)^{r+1} v^{r+1}}{(r+1)!} \tilde{I}'(v) 
         \,dv \biggr\} 
\\&= \E\biggl\{ \vecf{Z}_i \biggl[ -\tau
       +\sum_{k=0}^{r} F_{\Lambda|Z}^{(k)}(0\mid \vecf{Z}_i) \frac{(-h_n)^k}{k!}
         \overbrace{\int_{-1}^{1} v^k \tilde{I}'(v) \,dv }^{=0\textrm{ for $1\le k\le r-1$ by \cref{a:Itilde}}}
     \biggr] \biggr\} 
\\&\quad+ O(h_n^{r+1}) \overbrace{\E\biggl\{ \vecf{Z}_i 
        \int_{-1}^{1} 
        \overbrace{f_{\Lambda|Z}^{(r)}(-\tilde{h}v \mid \vecf{Z}_i)}^{\textrm{bounded by \cref{a:U}}}
        v^{r+1} \tilde{I}'(v) 
         \,dv \biggr\} }^{O(1)\textrm{ by \cref{a:U,a:Itilde}}}
\\&= \E\left\{ \vecf{Z}_i \left[ -\tau +F_{\Lambda|Z}(0\mid \vecf{Z}_i)
       +f_{\Lambda|Z}^{(r-1)}(0\mid \vecf{Z}_i) \frac{(-h_n)^r}{r!}
         \int_{-1}^{1} v^r \tilde{I}'(v) \,dv 
     \right] \right\} 
+ O(h_n^{r+1})
\\&= \E\left\{ \vecf{Z}_i \left[ -\tau +\E( \Ind{\Lambda_i\le0} \mid \vecf{Z}_i ) 
     \right] \right\} 
+ \overbrace{\frac{(-h_n)^r}{r!}}^{\textrm{$r$ is even}} \left[ \int_{-1}^{1} v^r \tilde{I}'(v) \,dv \right] 
  \E\left[ \vecf{Z}_i f_{\Lambda|Z}^{(r-1)}(0\mid \vecf{Z}_i) \right] 
+ O(h_n^{r+1})
\\&= \overbrace{\E\left\{ \E\left[ \vecf{Z}_i \left( \Ind{\Lambda_i\le0} -\tau \right) \mid \vecf{Z}_i  
     \right] \right\} }^{=\E\left[ \vecf{Z}_i \left( \Ind{\Lambda_i\le0} -\tau \right) \right]=0\textrm{ by \cref{a:B}}}
+ \frac{h_n^r}{r!} \left[ \int_{-1}^{1} v^r \tilde{I}'(v) \,dv \right] 
  \E\left[ \vecf{Z}_i f_{\Lambda|Z}^{(r-1)}(0\mid \vecf{Z}_i) \right] 
+ O(h^{r+1})
\\&= \frac{h_n^r}{r!} \left[ \int_{-1}^{1} v^r \tilde{I}'(v) \,dv \right] 
  \E\left[ \vecf{Z}_i f_{\Lambda|Z}^{(r-1)}(0\mid \vecf{Z}_i) \right] 
+ O(h_n^{r+1})
= O(h_n^r)
. 
\end{align*}
Thus, the result follows if $\sqrt{n}h_n^r=o(1)$, i.e., $h_n=o(n^{-1/(2r)})$ as in \cref{a:h}. \end{proof}

\vspace{0.5cm}

\begin{proof}[Proof of \Cref{thm:normality}]
We first establish the asymptotic normality of the smoothed MM estimator. 
We then prove the asymptotic normality of the smoothed GMM estimator. 

\paragraph{MM asymptotic normality.}
Recall from 
\cref{eqn:def-est-MM} that 
$\vecf{0} =
\hat{\vecf{M}}_n(\hat{\vecf{\beta}}_{\mathrm{MM}},\tau) $. 
Define 
\begin{equation}\label{eqn:nabla-Mn}
\vecf{\nabla}_{\vecf{\beta}\tr } \hat{\vecf{M}}_n(\vecf{\beta}_{0\tau},\tau) 
\equiv
\left. \pD{}{\vecf{\beta}\tr } \hat{\vecf{M}}_n(\vecf{\beta},\tau)
\right\rvert_{\vecf{\beta}=\vecf{\beta}_{0\tau}} . 
\end{equation}
Let $\hat{\vecf{M}}_n^{(k)}(\vecf{\beta},\tau)$ refer to the $k$th element in the vector $\hat{\vecf{M}}_n(\vecf{\beta},\tau)$, so $\vecf{\nabla}_{\vecf{\beta}\tr } \hat{\vecf{M}}_n^{(k)}(\vecf{\beta}_{0\tau},\tau)$ is a row vector and $\vecf{\nabla}_{\vecf{\beta}} \hat{\vecf{M}}_n^{(k)}(\vecf{\beta}_{0\tau},\tau)$ is a column vector. 
Define 
\begin{equation}\label{eqn:def-M-dot}
\dot{\matf{M}}_n(\tau)
\equiv \left( 
  \vecf{\nabla}_{\vecf{\beta}} \hat{\vecf{M}}_n^{(1)}(\tilde{\vecf{\beta}}_{(1)},\tau), 
  \ldots, 
  \vecf{\nabla}_{\vecf{\beta}} \hat{\vecf{M}}_n^{(d_\beta)}(\tilde{\vecf{\beta}}_{(d_\beta)},\tau)
\right) \tr  , 
\end{equation}
a $d_\beta\times d_\beta$ matrix with its first row equal to that of $\vecf{\nabla}_{\vecf{\beta}\tr } \hat{\vecf{M}}_n(\tilde{\vecf{\beta}}_{(1)},\tau)$, its second row equal to that of $\vecf{\nabla}_{\vecf{\beta}\tr } \hat{\vecf{M}}_n(\tilde{\vecf{\beta}}_{(2)},\tau)$, etc., where each vector $\tilde{\vecf{\beta}}_{(k)}$ lies on the line segment between $\vecf{\beta}_{0\tau}$ and $\hat{\vecf{\beta}}_{\mathrm{MM}}$. 
Due to smoothing, we can take a derivative (for any $n$) to obtain a mean value expansion, and then rearrange: 
\begin{gather}\label{eqn:Mn-expansion}
\vecf{0} = 
\hat{\vecf{M}}_n(\vecf{\beta}_{0\tau},\tau) 
+ \dot{\matf{M}}_n(\tau) ( \hat{\vecf{\beta}}_{\mathrm{MM}} - \vecf{\beta}_{0\tau} ) 
,\\\label{eqn:beta-hat-linear-form1}
\sqrt{n} ( \hat{\vecf{\beta}}_{\mathrm{MM}} - \vecf{\beta}_{0\tau} )
= -[ \dot{\matf{M}}_n(\tau) ]^{-1}
   \sqrt{n} \hat{\vecf{M}}_n(\vecf{\beta}_{0\tau},\tau) 
. 
\end{gather}
From \cref{a:G-est}, after plugging in definitions, $\dot{\matf{M}}_n(\tau) \pconv \matf{G}$; applying the continuous mapping theorem, 
$-[ \dot{\matf{M}}_n(\tau) ]^{-1} \pconv  -\matf{G}^{-1}$. 
Using \cref{a:CLT}, the rest of the right-hand side of \cref{eqn:beta-hat-linear-form1} has an asymptotic normal distribution. 
\Cref{eqn:beta-hat-linear-form1} also implies the asymptotically linear (influence function) representation 
\begin{equation}\label{eqn:asy-linear}
\sqrt{n} ( \hat{\vecf{\beta}}_{\mathrm{MM}} - \vecf{\beta}_{0\tau} )
= \frac{1}{\sqrt{n}} \sum_{i=1}^{n} 
[-\dot{\matf{M}}_n(\tau) ]^{-1}
   \vecf{g}_{ni}(\vecf{\beta}_{0\tau},\tau) 
= -\frac{1}{\sqrt{n}} \sum_{i=1}^{n} 
   \matf{G}^{-1}
   \vecf{g}_{ni}(\vecf{\beta}_{0\tau},\tau) 
  +o_p(1) . 
\end{equation}

Next, apply the continuous mapping theorem (CMT), using the nonsingularity of $\matf{G}$ assumed in \cref{a:U} and the result $\dot{\matf{M}}_n(\tau)\pconv \matf{G}$ in \cref{a:G-est} to obtain 
$[ \dot{\matf{M}}_n(\tau) ]^{-1} \pconv  \matf{G}^{-1}$. 
Using the CMT again, combine this with the results in \cref{eqn:beta-hat-linear-form1} and \cref{lem:Mn0-normality}: 
\begin{align*}
\sqrt{n} ( \hat{\vecf{\beta}}_{\mathrm{MM}} - \vecf{\beta}_{0\tau} )
&= -\overbrace{[ \dot{\matf{M}}_n(\tau) ]^{-1}}^{\textrm{use \cref{a:G-est} and CMT}}
    \overbrace{\sqrt{n} \hat{\vecf{M}}_n(\vecf{\beta}_{0\tau},\tau)}^{\textrm{use \cref{lem:Mn0-normality}}} 
\\&\dconv  -\matf{G}^{-1}
    \Normalp{\vecf{0}}{\matf{\Sigma}_{\tau}} 
  \stackrel{d}{=} \Normalp{\vecf{0}}{\matf{G}^{-1} \matf{\Sigma}_{\tau} [\matf{G}\tr ]^{-1}} 
. 
\end{align*}

\paragraph{GMM asymptotic normality.} For GMM, the approach is similar, but starting from the first-order condition for the mean value expansion. 
From the definition of $\hat{\vecf{\beta}}_{\mathrm{GMM}}$ in \cref{eqn:def-est-GMM}, we have the first-order condition 
\begin{equation}\label{eqn:FOC}
\left[ \vecf{\nabla}_{\vecf{\beta}\tr} \hat{\vecf{M}}_n(\hat{\vecf{\beta}}_{\mathrm{GMM}},\tau) \right]\tr \hat{\matf{W} }   \hat{\vecf{M}}_n(\hat{\vecf{\beta}}_{\mathrm{GMM}},\tau)
= \vecf{0} . 
\end{equation}
We reuse the notation from \cref{eqn:nabla-Mn,eqn:def-M-dot}, but now $\tilde{\vecf{\beta}}_{(k)}$ lies between $\vecf{\beta}_{0\tau}$ and $\hat{\vecf{\beta}}_{\mathrm{GMM}}$. 
By the mean value theorem, 
\begin{equation}\label{eqn:MVT}
\hat{\vecf{M}}_n(\hat{\vecf{\beta}}_{\mathrm{GMM}},\tau)=
\hat{\vecf{M}}_n(\vecf{\beta}_{0\tau},\tau)
+\dot{\matf{M}}_n(\tau) ( \hat{\vecf{\beta}}_{\mathrm{GMM}} - \vecf{\beta}_{0\tau} ).
\end{equation}
Pre-multiplying \cref{eqn:MVT} by $[ \vecf{\nabla}_{\vecf{\beta}\tr} \hat{\vecf{M}}_n(\hat{\vecf{\beta}}_{\mathrm{GMM}},\tau) ]\tr \hat{\matf{W} }$ and using \cref{eqn:FOC} for the first equality, 
\begin{align*}
\vecf{0} 
&=[ \vecf{\nabla}_{\vecf{\beta}\tr} \hat{\vecf{M}}_n(\hat{\vecf{\beta}}_{\mathrm{GMM}},\tau) ]\tr \hat{\matf{W} }   \hat{\vecf{M}}_n(\hat{\vecf{\beta}}_{\mathrm{GMM}},\tau)\\
&= [ \vecf{\nabla}_{\vecf{\beta}\tr} \hat{\vecf{M}}_n(\hat{\vecf{\beta}}_{\mathrm{GMM}},\tau) ]\tr \hat{\matf{W} }
\hat{\vecf{M}}_n(\vecf{\beta}_{0\tau},\tau)
+[ \vecf{\nabla}_{\vecf{\beta}\tr} \hat{\vecf{M}}_n(\hat{\vecf{\beta}}_{\mathrm{GMM}},\tau) ]\tr \hat{\matf{W} }
\dot{\matf{M}}_n(\tau) ( \hat{\vecf{\beta}}_{\mathrm{GMM}} - \vecf{\beta}_{0\tau} ).
\end{align*}
Multiplying by $\sqrt{n}$ and rearranging, 
\begin{align}\notag
& \sqrt{n} ( \hat{\vecf{\beta}}_{\mathrm{GMM}} - \vecf{\beta}_{0\tau} )
\\[-12pt]&= -\bigl\{ [ \vecf{\nabla}_{\vecf{\beta}\tr} \hat{\vecf{M}}_n(\hat{\vecf{\beta}}_{\mathrm{GMM}},\tau) ]\tr \hat{\matf{W} }
\dot{\matf{M}}_n(\tau) \bigr\}^{-1}
\left[ \vecf{\nabla}_{\vecf{\beta}\tr} \hat{\vecf{M}}_n(\hat{\vecf{\beta}}_{\mathrm{GMM}},\tau) \right]\tr \hat{\matf{W} }
\overbrace{\sqrt{n} \hat{\vecf{M}}_n(\vecf{\beta}_{0\tau},\tau)}^{=O_p(1)}
\label{eqn:est-GMM-asy-linear-almost}
\\&= -\{ \matf{G}\tr \matf{W} \matf{G} \}^{-1}
\matf{G}\tr \matf{W}
\sqrt{n} \hat{\vecf{M}}_n(\vecf{\beta}_{0\tau},\tau)
+o_p(1) ,
\label{eqn:est-GMM-asy-linear}
\end{align}
where $\hat{\matf{W}} = \matf{W} +o_p(1)$ by \cref{a:W}, 
and $\dot{\matf{M}}_n(\tau) = \matf{G} +o_p(1)$ and 
$\vecf{\nabla}_{\vecf{\beta}\tr} \hat{\vecf{M}}_n(\hat{\vecf{\beta}}_{\mathrm{GMM}},\tau) 
= \matf{G} +o_p(1)$ 
by \cref{a:G-est}. 
From \cref{lem:Mn0-normality}, 
$\sqrt{n}\hat{\vecf{M}}_n(\vecf{\beta}_{0\tau},\tau)
\dconv 
\Normalp{\vecf{0}}{\matf{\Sigma}_{\tau}}$. 
Applying the continuous mapping theorem yields 
the stated result. 
\end{proof}

\section{Primitive conditions for high-level assumptions}
\label{sec:app-primitive}

The following subsections discuss primitive conditions for the high-level \Cref{a:ULLN,a:G-est,a:CLT}.

\subsection{\texorpdfstring{\Cref{a:ULLN}}{Assumption \ref{a:ULLN}}}\label{sec:app-primitive-ULLN}

\Cref{a:ULLN} is a high-level ULLN-type assumption. 
Intuitively, it holds under weak enough dependence and a moment restriction on $\vecf{Z}_i$. 
Howver, it is not trivial since most ULLNs assume a constant function $\vecf{g}(\cdot)$ instead of a function indexed by $n$. 
We provide an example of sufficient lower-level assumptions in \cref{lem:ULLN}.

\begin{lemma}\label{lem:ULLN}
Let \Cref{a:XYZ,a:B,a:Z,a:Itilde,a:Lambda,a:Y} hold. 
Additionally, assume the following. 
(i) $\bigl(\mathcal{B},d(\cdot)\bigr)$ is a metric space. 
(ii) Defining open balls $B(\vecf{\beta},\rho)\equiv\{\tilde{\vecf{\beta}}\in\mathcal{B}:d(\vecf{\beta},\tilde{\vecf{\beta}})<\rho\}$, 
\begin{equation}\begin{split}\label{eqn:g-star}
\vecf{g}_n^*(\vecf{Y}_i,\vecf{X}_i,\vecf{Z}_i,\vecf{\beta},\tau, \rho)
&\equiv \sup\left\{ \vecf{g}_n(\vecf{Y}_i,\vecf{X}_i,\vecf{Z}_i,\tilde{\vecf{\beta}},\tau) : \tilde{\vecf{\beta}}\in B(\vecf{\beta},\rho) \right\}
,\\ 
\vecf{g}_{*n}(\vecf{Y}_i,\vecf{X}_i,\vecf{Z}_i,\vecf{\beta},\tau, \rho)
&\equiv \inf\left\{ \vecf{g}_n(\vecf{Y}_i,\vecf{X}_i,\vecf{Z}_i,\tilde{\vecf{\beta}},\tau) : \tilde{\vecf{\beta}}\in B(\vecf{\beta},\rho) \right\} 
\end{split}\end{equation}
are random variables for all $i$, $\vecf{\beta}\in\mathcal{B}$, and sufficiently small $\rho$ (which may depend on $\vecf{\beta}$), 
where the $\sup$ and $\inf$ are taken separately for each element of the vector. 
(iii) A pointwise WLLN holds for the random vectors in \cref{eqn:g-star}, for each $\vecf{\beta}\in\mathcal{B}$. 
(iv) The data are strictly stationary. 
Then, for a fixed $\tau\in(0,1)$, using the definition in \cref{eqn:def-M-hat}, 
\begin{equation*}
\sup_{\vecf{\beta}\in\mathcal{B}}
\absbig{ \hat{\vecf{M}}_n(\vecf{\beta},\tau) - \E[ \hat{\vecf{M}}_n(\vecf{\beta},\tau) ] }
= o_p(1) . 
\end{equation*}
\end{lemma}
\begin{proof}
We show that the theorem in \citet{Andrews87} applies. 
The theorem concerns a uniform law of large numbers (ULLN) for a sample average of functions of the data. 
By Comment 6 in \citet{Andrews87}, both the data and the functions may be indexed by both $i$ and $n$. 
In our case, the function $\vecf{g}_n(\cdot)$ is not indexed by $i$ but must be indexed by $n$ since it depends on the sequence of bandwidths, $h_n$. 
We continue to index the observations only by $i$ but note that triangular arrays are permitted by \citet{Andrews87}. 
Since \citet{Andrews87} presumes a scalar-valued function, we write $g_n(\cdot)$; since the dimension of $\vecf{g}_n(\cdot)$ is fixed and finite, uniformity extends immediately to the vector. 

Assumption A1 in \citet{Andrews87} is simply that $\mathcal{B}$ is compact, which is in our \cref{a:B}. 
(More recent work shows that ``compact'' can be replaced by ``totally bounded'' under a metric; see \citet{Andrews92} and \citet{PotscherPrucha94}.) 

Assumption A2(a) in \citet{Andrews87} is a technical measurability assumption; this is assumption (ii) in the statement of \cref{lem:ULLN}. 

Assumption A2(b) in \citet{Andrews87} is assumption (iii) in the statement of the lemma. 
There are many WLLNs for weakly dependent triangular arrays, where dependence is quantified and restricted in various ways; for example, see Theorem 2 in \citet{Andrews88} and the theorems in \citet{deJong98}.  
With iid sampling, sufficient primitive conditions for a WLLN are already in our \cref{a:Z} and \cref{a:Itilde}, respectively: a) $\E\bigl(\norm{\vecf{Z}_i}^2\bigr)<\infty$, and b) $\tilde{I}(\cdot)$ is bounded. 
From \cref{a:Itilde}, $-2\le \tilde{I}(\cdot)-\tau\le 2$, so we have the dominating function $\abs{\vecf{g}_n(\vecf{Y}_i,\vecf{X}_i,\vecf{Z}_i,\vecf{\beta},\tau)} \le 2\abs{\vecf{Z}_i}$. 
Consequently, 
\begin{equation*}
\abs{ \vecf{g}_n^*(\vecf{Y}_i,\vecf{X}_i,\vecf{Z}_i,\vecf{\beta},\tau, \rho) } \le 2\abs{\vecf{Z}_i} , 
\quad
\abs{ \vecf{g}_{*n}(\vecf{Y}_i,\vecf{X}_i,\vecf{Z}_i,\vecf{\beta},\tau, \rho) } \le 2\abs{\vecf{Z}_i} . 
\end{equation*}
If the data are iid, then $\vecf{g}_n^*(\vecf{Y}_i,\vecf{X}_i,\vecf{Z}_i,\vecf{\beta},\tau, \rho)$ is a row-wise iid triangular array. 
Thus, a sufficient condition for a WLLN is $\sup_n \E\bigl[\norm{\vecf{g}_n^*}^2\bigr]<\infty$ (as can be shown with Markov's inequality). 
This condition holds since $\sup_n \E\bigl[\norm{\vecf{g}_n^*}^2\bigr] \le \E\bigl[\norm{2\vecf{Z}_i}^2\bigr] < \infty$ by \cref{a:Z}. 
An extension to independent but not identical sampling follows from a Lindeberg condition for $\vecf{Z}_i$. 
A pointwise WLLN continues to hold with dependence, too, as long as the dependence is not too strong. 

Assumption A3 in \citet{Andrews87} in our notation is 
\begin{equation}\label{eqn:Andrews87-A3}
\lim_{\rho\to0} \sup_{n\ge1} \absbig{
\frac{1}{n} \sum_{i=1}^{n} \left\{ 
  \E\left[ \vecf{g}_n^*(\vecf{Y}_i,\vecf{X}_i,\vecf{Z}_i,\vecf{\beta},\tau, \rho) \right] 
- \E\left[ \vecf{g}_n(\vecf{Y}_i,\vecf{X}_i,\vecf{Z}_i,\vecf{\beta},\tau) \right] 
\right\} } = \vecf{0} , 
\end{equation}
and similarly when replacing $\vecf{g}_n^*$ with $\vecf{g}_{*n}$. 
Since $\vecf{g}_n$ varies with $n$ but not $i$, the strict stationarity in assumption (iv) in \cref{lem:ULLN} implies the summands do not vary with $i$, which simplifies \cref{eqn:Andrews87-A3} to be 
\begin{equation}\label{eqn:Andrews-A3-simple}
\lim_{\rho\to0} \sup_{n\ge1} \abs{ \vecf{\Delta}_n } = \vecf{0} , \;
\vecf{\Delta}_n \equiv 
  \E\left[ \vecf{g}_n^*(\vecf{Y}_i,\vecf{X}_i,\vecf{Z}_i,\vecf{\beta},\tau, \rho) \right] 
- \E\left[ \vecf{g}_n(\vecf{Y}_i,\vecf{X}_i,\vecf{Z}_i,\vecf{\beta},\tau) \right] , \;
\vecf{\Delta}_\infty \equiv \lim_{n\to\infty} \vecf{\Delta}_n . 
\end{equation}
Strict stationarity is not necessary, though, as long as \cref{eqn:Andrews87-A3} still holds. 

A necessary condition for \cref{eqn:Andrews-A3-simple} is pointwise convergence $\lim_{\rho\to0}\vecf{\Delta}_n=\vecf{0}$ for each fixed $n$. 
By \cref{a:Itilde,a:B}, $\vecf{g}_{ni}(\vecf{\beta},\tau)$ is continuous and even differentiable in $\vecf{\beta}$. 
Additionally, as noted above, $2\absbig{\vecf{Z}_i}$ is a dominating function with finite expectation (by \cref{a:Z}), so the dominated convergence theorem gives 
\begin{align*}
\lim_{\rho\to0} \vecf{\Delta}_n 
&= \lim_{\rho\to0} 
\E\left[ \vecf{g}_n^*(\vecf{Y}_i,\vecf{X}_i,\vecf{Z}_i,\vecf{\beta},\tau, \rho) 
        -\vecf{g}_n(\vecf{Y}_i,\vecf{X}_i,\vecf{Z}_i,\vecf{\beta},\tau) \right] 
\\&=
\E\left\{
  \lim_{\rho\to0} 
  \left[ \vecf{g}_n^*(\vecf{Y}_i,\vecf{X}_i,\vecf{Z}_i,\vecf{\beta},\tau, \rho) 
        -\vecf{g}_n(\vecf{Y}_i,\vecf{X}_i,\vecf{Z}_i,\vecf{\beta},\tau) \right] 
\right\}
\\&= \E\{\vecf{0} \}
= \vecf{0} , 
\end{align*}
and similarly for $\vecf{g}_{*n}$, for any $n$. 

For $\vecf{\Delta}_\infty$, as $n\to\infty$, we can again move the limit inside expectations by the dominated convergence theorem, so 
\begin{align*}
\lim_{n\to\infty} 
\E\left[ \vecf{g}_n(\vecf{Y}_i,\vecf{X}_i,\vecf{Z}_i,\vecf{\beta},\tau) \right]
&= \E\left\{
  \lim_{n\to\infty} 
  \vecf{g}_n(\vecf{Y}_i,\vecf{X}_i,\vecf{Z}_i,\vecf{\beta},\tau) 
\right\}
\\&= \E\left\{
  \vecf{Z}_i\left[ \Ind{\Lambda(\vecf{Y}_i,\vecf{X}_i,\vecf{\beta})\le0} - \tau \right] 
\right\} 
\\&= \E\left\{ 
  \E\left[ \vecf{Z}_i\left( \Ind{\Lambda(\vecf{Y}_i,\vecf{X}_i,\vecf{\beta})\le0} - \tau \right)
          \mid \vecf{Z}_i \right] 
\right\} 
\\&= \E\left\{ \vecf{Z}_i \left[ 
  \E\left( \Ind{\Lambda(\vecf{Y}_i,\vecf{X}_i,\vecf{\beta})\le0} 
          \mid \vecf{Z}_i \right) - \tau \right] 
\right\} 
\\&= \E\left\{ \vecf{Z}_i \left[ 
  \Pr\left( \Lambda(\vecf{Y}_i,\vecf{X}_i,\vecf{\beta}) \le 0 
          \mid \vecf{Z}_i \right) - \tau \right] 
\right\} 
. 
\end{align*}
Technically, since $\tilde{I}(0/h_n)=0.5$ for any $h_n>0$, the function $\tilde{I}(\cdot/h_n)\to\Ind{\cdot\ge0}-0.5\Ind{\cdot=0}$ as $n\to\infty$, so we have 
\begin{align*}
& \E\left\{ \vecf{Z}_i \left[ 
  \E\left( \Ind{\Lambda(\vecf{Y}_i,\vecf{X}_i,\vecf{\beta})\le0} 
          -0.5 \Ind{\Lambda(\vecf{Y}_i,\vecf{X}_i,\vecf{\beta})=0} 
          \mid \vecf{Z}_i \right) - \tau \right] 
\right\} 
\\&= \E\left\{ \vecf{Z}_i \left[ 
  \Pr\left( \Lambda(\vecf{Y}_i,\vecf{X}_i,\vecf{\beta}) \le 0 
          \mid \vecf{Z}_i \right) 
  - 0.5 \overbrace{\Pr\left( \Lambda(\vecf{Y}_i,\vecf{X}_i,\vecf{\beta}) = 0 
          \mid \vecf{Z}_i \right) }^{=0\textrm{ a.s. by \cref{a:Y}}}
          - \tau \right] 
\right\} . 
\end{align*}
That is, by \cref{a:Y}, the $0.5$ adjustment corresponds to a zero probability event that does not affect the overall expectation. 
For $\vecf{g}_n^*$, similarly, 
\begin{align*}
\lim_{n\to\infty} 
\E\left[ \vecf{g}_n^*(\vecf{Y}_i,\vecf{X}_i,\vecf{Z}_i,\vecf{\beta},\tau,\rho) \right]
&= \E\left\{ \sup_{\tilde{\vecf{\beta}}\in B(\vecf{\beta},\rho)}
               \vecf{Z}_i \left[ 
  \Pr\left( \Lambda(\vecf{Y}_i,\vecf{X}_i,\tilde{\vecf{\beta}}) \le 0 
          \mid \vecf{Z}_i \right) - \tau \right] 
\right\} 
. 
\end{align*}
Consequently, 
\begin{equation*}
\vecf{\Delta}_\infty 
= \E\left\{ \vecf{Z}_i \left[ 
  \Pr\left( \Lambda(\vecf{Y}_i,\vecf{X}_i,\vecf{\beta}) \le 0 
          \mid \vecf{Z}_i \right) - \tau \right] 
- \sup_{\tilde{\vecf{\beta}}\in B(\vecf{\beta},\rho)}
               \vecf{Z}_i \left[ 
  \Pr\left( \Lambda(\vecf{Y}_i,\vecf{X}_i,\tilde{\vecf{\beta}}) \le 0 
          \mid \vecf{Z}_i \right) - \tau \right] 
\right\} . 
\end{equation*}
For this to have a limit of zero as $\rho\to0$ again requires continuity in $\rho$, but the necessary and sufficient conditions are different than for fixed $n$. 
Sufficient conditions here are found in \cref{a:Lambda,a:Y}: $\Lambda(\cdot)$ is continuous in $\vecf{\beta}$, and for any $\vecf{\beta}\in\mathcal{B}$ and almost all $\vecf{z}\in\mathcal{Z}$, the conditional distribution of $\Lambda(\vecf{Y}_i,\vecf{X}_i,\vecf{\beta})$ given $\vecf{Z}_i=\vecf{z}$ is continuous in a neighborhood of zero. 
For example, if $\Lambda(Y_i,\vecf{X}_i,\vecf{\beta})=Y_i-\vecf{X}_i\tr \vecf{\beta}$, then it is sufficient that $Y_i$ has a continuous distribution given almost all $\vecf{Z}_i=\vecf{z}$. 

Given $\lim_{\rho\to0}\vecf{\Delta}_n=\vecf{0}$ for any $n<\infty$ and $n=\infty$, the conclusion $\lim_{\rho\to0} \sup_{n\ge1} \absbig{ \vecf{\Delta}_n } = \vecf{0}$ follows because 
the supremum is attained: $\sup_{n\ge1} \absbig{ \vecf{\Delta}_n } = \absbig{\vecf{\Delta}_k}$ for some $k\ge1$ or $k=\infty$. 
If instead $\lim_{\rho\to0}\vecf{\Delta}_n=\vecf{0}$ only for $n\ge1$, and not with $\lim_{n\to\infty}$, then it would be possible for all $\lim_{\rho\to0}\vecf{\Delta}_n=\vecf{0}$ pointwise but $\lim_{\rho\to0}\sup_{n\ge1}\vecf{\Delta}_n\ne\vecf{0}$; for example if $\Delta_n=(1-1/n)^{1/\rho}$ then all $\lim_{\rho\to0}\Delta_n=0$ but $\sup_{n\ge1}\Delta_n=1$ for any $\rho$, so $\lim_{\rho\to0}\sup_{n\ge1}\Delta_n=1$. 
This is why the calculations for $\vecf{\Delta}_\infty$ are necessary. 

Having verified A1, A2, and A3 in \citet{Andrews87}, his theorem applies, yielding the desired ULLN. 
\end{proof}

\subsection{\texorpdfstring{\Cref{a:CLT}}{Assumption \ref{a:CLT}}}
\label{sec:app-primitive-CLT}

To establish \Cref{a:CLT}, with iid data, the Lindeberg--Feller CLT can be applied as in the proof of Theorem 1 in \citet{KaplanSun17}. 
More generally, \cref{a:CLT} can hold under weak dependence. 
For example, Theorem 3.13 in \citet[Ch.\ 2]{Wooldridge86}, as reproduced in Proposition 1 of \citet{Andrews91CLT}, is a CLT for near epoch dependent triangular arrays that holds under some moment and dependence restrictions. 
The moment restriction, condition (ii), in our notation is $\E\{ \normbig{\vecf{g}_{ni}(\vecf{\beta}_{0\tau},\tau)}^{2+\epsilon} \} < \infty$ for some $\epsilon>0$. 
Since $\absbig{\vecf{g}_{ni}(\vecf{\beta}_{0\tau},\tau)}<2\absbig{\vecf{Z}_i}$, if the underlying $\vecf{Z}_i$ are strictly stationary then $\E( \normbig{\vecf{Z}_i}^{2+\epsilon} ) < \infty$ is sufficient; triangular array data are allowed if $\sup_{i\le n,n\ge1}\E( \normbig{\vecf{Z}_{ni}}^{2+\epsilon} ) < \infty$ for some $\epsilon>0$.

\subsection{\texorpdfstring{\Cref{a:G-est}}{Assumption \ref{a:G-est}}}
\label{sec:app-primitive-G-est}

Unfortunately, for multiple reasons, \Cref{a:G-est} cannot be deduced simply by applying a result like Lemma 4.3 in \citet[p.\ 2156]{NeweyMcFadden94}. 
Fortunately, it is closely related to the well-studied result of consistency of the kernel estimator for the quantile regression asymptotic covariance matrix. 
Since the argument is the same for each row in the matrix, we consider row $k$. 
Plugging in definitions, 
\begin{align}\notag
\dot{\vecf{M}}^{(k,\cdot)}_n(\tau) 
&= \nabla_{\vecf{\beta}\tr } \hat{M}_n^{(k)}(\tilde{\vecf{\beta}},\tau) 
= \frac{1}{n} \sum_{i=1}^{n} \left. \pD{}{\vecf{\beta}\tr } g_{ni}^{(k)}(\vecf{\beta},\tau) \right\rvert_{\vecf{\beta}=\tilde{\vecf{\beta}}} 
\\&= \frac{1}{n} \sum_{i=1}^{n} Z_i^{(k)} 
  \tilde{I}'\bigl( -\Lambda(\vecf{Y}_i,\vecf{X}_i,\tilde{\vecf{\beta}})/h_n \bigr) 
  (-h_n^{-1}) 
  \vecf{D}_i(\tilde{\vecf{\beta}})\tr ,  
\label{eqn:G-est-k}
\\
\vecf{D}_i(\vecf{b})
&\equiv \left. \pD{}{\vecf{\beta} } \Lambda(\vecf{Y}_i,\vecf{X}_i,\vecf{\beta}) \right\rvert_{\vecf{\beta}=\vecf{b}} . 
\label{eqn:def-Di}
\end{align}
By \cref{a:Itilde}, $\tilde{I}'(\cdot)$ is a kernel function. 
The RHS of \cref{eqn:G-est-k} is closely related to the kernel estimator of the usual quantile regression estimator's asymptotic covariance matrix initially proposed under censoring by \citet[eqn.\ (5.6)]{Powell84} and without censoring in \citet{Powell91}, but with two differences: 1) we have $\tilde{\vecf{\beta}}$ instead of $\hat{\vecf{\beta}}$, 2) we have a more general model. 
As a special case of our model, the ``usual'' quantile regression model would have $\vecf{Z}_i=\vecf{X}_i$ and $\Lambda(Y_i,\vecf{X}_i,\vecf{\beta})=Y_i-\vecf{X}_i\tr \vecf{\beta}$, so $\nabla_{\vecf{\beta}}\Lambda(Y_i,\vecf{X}_i,\tilde{\vecf{\beta}})=-\vecf{X}_i$, and \cref{eqn:G-est-k} simplifies to 
\begin{equation}\label{eqn:G-est-QR}
\frac{1}{n} \sum_{i=1}^{n} X_i^{(k)} 
  \tilde{I}'\bigl( (\vecf{X}_i\tr \tilde{\vecf{\beta}}-Y_i)/h_n \bigr) 
  (-h_n^{-1}) 
  (-\vecf{X}_i\tr )
= \frac{1}{nh_n} \sum_{i=1}^{n} 
  \tilde{I}'\left( \frac{Y_i-\vecf{X}_i\tr \tilde{\vecf{\beta}}}{h_n} \right) 
  X_i^{(k)} \vecf{X}_i\tr  , 
\end{equation}
using the symmetry $\tilde{I}'(-u)=\tilde{I}'(u)$ from \cref{a:Itilde}. 
Since $\tilde{\vecf{\beta}}$ lies between $\vecf{\beta}_{0\tau}$ and $\hat{\vecf{\beta}}_{\mathrm{MM}}$, proofs using $\hat{\vecf{\beta}}_{\mathrm{MM}}$ still hold since $\sqrt{n}$-consistency of $\hat{\vecf{\beta}}_{\mathrm{MM}}$ implies $\sqrt{n}$-consistency of $\tilde{\vecf{\beta}}$. 

For \cref{eqn:G-est-QR}, \citet{Kato12} shows consistency (i.e., our \cref{a:G-est}) with both iid and weakly dependent data. 
In fact, he shows the stronger result of asymptotic normality, so some of his assumptions may be weakened if only consistency is required; for example, his $h_n\sqrt{n}/\log(n)\to\infty$ in Assumption 13 can be (slightly) weakened to $h_n\sqrt{n}\to\infty$, 
and not as many moments of $\vecf{X}_i$ are required. 
Specifically, \citet{Kato12} considers strictly stationary $\beta$-mixing data, and the mixing coefficients $\beta(j)$, moments of $\vecf{X}_i$, and bandwidth rate are jointly restricted in his Assumptions 9, 10, and 13 (p.\ 268): $\sum_{j=1}^{\infty} j^\lambda [\beta(j)]^{1-2/\delta} < \infty$ for some $\delta>2$ and $\lambda > 1-2/\delta$; $\E[\norm{\vecf{X}_i}^{\max\{6,2\delta\}}]<\infty$; and for some integer sequence $s_n \to \infty$ with $s_n = o(\sqrt{nh_n})$, $(n/h_n)^{1/2} \beta(s_n) \to 0$. 

For the more general \cref{eqn:G-est-k}, similar conditions are sufficient if $\vecf{D}_i(\vecf{\beta})=\vecf{D}_i$, a random variable depending on $\vecf{Y}_i$ and $\vecf{X}_i$ but not the argument $\vecf{\beta}$. 
This occurs if (and only if) the residual function $\Lambda(\cdot)$ is linear-in-parameters. 
Then, $\vecf{D}_i$ replaces one of the $\vecf{X}_i$ in \citet{Kato12}, while $\vecf{Z}_i$ replaces the other. 
The most notable restriction is on moments of 
$\vecf{D}_i$ (which implies a certain number of finite moments for $\vecf{Y}_i$ and $\vecf{X}_i$) in addition to $\vecf{Z}_i$ (which is already in \cref{a:Z}). 
Many economic variables are bounded or reasonably have infinite moments (e.g., a normal distribution), in which case such moment assumptions are not binding. 
If $\vecf{D}_i(\vecf{b})$ does depend on its argument, then an extension of the argument itself in \citet{Kato12} is necessary. 

In \cref{eqn:G-est-QR}, $\vecf{X}_i$ plays the roles of both the derivative of $\Lambda(Y_i,\vecf{X}_i,\vecf{\beta})$ and the instrument vector, so the PDF in $\matf{G}$ is just conditional on $\vecf{X}_i$; more generally, both the instrument vector and derivative must be conditioned on. 
This can be seen by computing the expectation of \cref{eqn:G-est-k} in a similar manner to the proof of \cref{lem:Mn0-normality}. 
After replacing $\tilde{\vecf{\beta}}=\vecf{\beta}_{0\tau}+O_p(n^{-1/2})$ and dropping the remainder, letting $\vecf{D}_i\equiv \nabla_{\vecf{\beta}}\Lambda(\vecf{Y}_i,\vecf{X}_i,\vecf{\beta}_{0\tau})$ and $\Lambda_i\equiv \Lambda(\vecf{Y}_i,\vecf{X}_i,\vecf{\beta}_{0\tau})$, 
\begin{align*}
\E[ \dot{\vecf{M}}^{(k,\cdot)}_n(\tau)  ]
&\doteq \frac{1}{n} \sum_{i=1}^{n} 
   \overbrace{\E\biggl\{ 
    Z_i^{(k)} 
    \tilde{I}'\bigl( -\Lambda(\vecf{Y}_i,\vecf{X}_i,\vecf{\beta}_{0\tau})/h_n \bigr) 
    (-h_n^{-1}) 
    \pD{}{\vecf{\beta}\tr } \Lambda(\vecf{Y}_i,\vecf{X}_i,\vecf{\beta}) \Bigr\rvert_{\vecf{\beta}=\tilde{\vecf{\beta}}} 
   \biggr\}}^{\textrm{same for all $i$ by \cref{a:XYZ}}}
\\&= \E\left[ Z_i^{(k)} \vecf{D}_i\tr 
       (-h_n^{-1}) \tilde{I}'\bigl( -\Lambda_i/h_n \bigr) 
     \right]
\\&= \E\left\{ Z_i^{(k)} \vecf{D}_i\tr  \E\left[
       (-h_n^{-1}) \tilde{I}'\bigl( -\Lambda_i/h_n \bigr) 
     \mid \vecf{Z}_i,\vecf{D}_i \right] \right\}
\\&= \E\left\{ Z_i^{(k)} \vecf{D}_i\tr  
     \int
       (-h_n^{-1}) \tilde{I}'\bigl( -L/h_n \bigr) 
       f_{\Lambda|\vecf{Z},\vecf{D}}(L\mid \vecf{Z}_i,\vecf{D}_i)\,dL 
     \right\}
\\&= \E\left\{ Z_i^{(k)} \vecf{D}_i\tr  
     \int_{-1}^{1}
       -\tilde{I}'(v) 
       f_{\Lambda|\vecf{Z},\vecf{D}}(-h_n v\mid \vecf{Z}_i,\vecf{D}_i)\,dv 
     \right\}
\\&= -\E\biggl\{ Z_i^{(k)} \vecf{D}_i\tr  
     \int_{-1}^{1}
       \tilde{I}'(v) 
       \bigl[ f_{\Lambda|\vecf{Z},\vecf{D}}(0\mid \vecf{Z}_i,\vecf{D}_i)
\\&\qquad\qquad\qquad\qquad\qquad\qquad
             -f_{\Lambda|\vecf{Z},\vecf{D}}'(-h_nv\mid \vecf{Z}_i,\vecf{D}_i) h_n v
             +\cdots \bigr]\,dv 
     \biggr\}
\\&= -\E\left[ Z_i^{(k)} \vecf{D}_i\tr  
       f_{\Lambda|\vecf{Z},\vecf{D}}(0\mid \vecf{Z}_i,\vecf{D}_i)
     \right]
\\&\quad
- h_n^r \E\biggl\{ Z_i^{(k)} \vecf{D}_i\tr  
     \int_{-1}^{1}
       \tilde{I}'(v) (v^r/r!)
       f_{\Lambda|\vecf{Z},\vecf{D}}^{(r)}(-\tilde{h}v \mid \vecf{Z}_i,\vecf{D}_i)
     \,dv 
     \biggr\}
\\&= -\E\left[ Z_i^{(k)} \vecf{D}_i\tr  
       f_{\Lambda|\vecf{Z},\vecf{D}}(0\mid \vecf{Z}_i,\vecf{D}_i)
     \right]
+ O(h_n^r) 
. 
\end{align*}

Finally, note that \citet{Kato12} does not require the conditional quantile regression model to be true, i.e., his results hold under misspecification; see his remark on p.\ 263 and Assumptions 8--13.

\section{Computational details}
\label{sec:app-comp}

\subsection{Simulation {DGP} details}
\label{sec:app-comp-DGPs}

\noindent \textbf{DGP 1}~
With iid sampling, there is a randomized treatment offer (instrument) 
$Z_i=1$ with probability $1/2$ and $Z_i=0$ otherwise; 
$U_i\sim\UnifDist(0,1)$ (the unobservable); 
$D_i=1$ if $i$ is treated and $D_i=0$ otherwise, with $\Pr(D_i=1\mid Z_i,U_i)=Z_i\min\{1,(4/3)U_i\}$ so that there is imperfect compliance and endogenous self-selection into treatment; 
and $Y_i=\beta(U_i)+D_i\gamma(U_i)$, where the function $\beta(\tau)=60+Q(\tau)$ with $Q(\cdot)$ the quantile function of the $\chi^2_3$ distribution, and $\gamma(\tau)=100(\tau-0.5)$, so there is heterogeneity of the quantile treatment effects.

\noindent \textbf{DGP 2}~
This DGP is a stationary time series regression with measurement error. 
The latent explanatory variable is $Z_t$, with $Z_0\sim\Normalp{0}{1}$, $Z_t=\rho_Z Z_{t-1} +\sqrt{1-\rho_Z^2} \nu_t$, and $\nu_t\stackrel{iid}{\sim}\Normalp{0}{1}$, so $\Varp{Z_t}=1$ for all $t$; we use $\rho_Z=0.5$. 
The measurement error is $\eta_t\stackrel{iid}{\sim}\Normalp{0}{1}$, and $X_t=Z_t+\eta_t$ is the observed (mismeasured) explanatory variable. 
Since the $\eta_t$ are independent, the lagged $X_{t-1}$ provides a valid instrument. 
The outcome is $Y_t=\gamma Z_t+\epsilon_t$ with $\gamma=1$ and $\epsilon_t$ unobserved. 
Finally, $\epsilon_t=\rho_\epsilon \epsilon_{t-1} + \sqrt{1-\rho_\epsilon^2} V_t$, $\epsilon_0\sim\Normalp{0}{1}$, $V_t\stackrel{iid}{\sim}\Normalp{0}{1}$, so the marginal distribution is $\epsilon_t\sim\Normalp{0}{1}$ for all $t$, and the series $\{\epsilon_t\}$, $\{\eta_t\}$, and $\{\nu_t\}$ are mutually independent. 
Letting $\beta(\tau)$ be the $\tau$-quantile of $\epsilon_t-\gamma\eta_t$, since $Y_t=\gamma Z_t+\epsilon_t=\gamma X_t +\epsilon_t-\gamma\eta_t$, we have the quantile restrictions $\Pr(Y_t-\gamma X_t-\beta(\tau) \le 0) = \Pr(\epsilon_t-\gamma\eta_t \le \beta(\tau))=\tau$, and $\Pr(\epsilon_t-\gamma\eta_t \le \beta(\tau) \mid X_{t-1})=\Pr(\epsilon_t-\gamma\eta_t \le \beta(\tau))=\tau$ since $X_{t-1}\independent \eta_t,\epsilon_t$, so the IVQR intercept and slope parameters are $\beta(\tau)$ and $\gamma=1$, respectively.

\noindent \textbf{DGP 3}~
This DGP is identical to DGP 2 except that 
$\epsilon_t=\rho_\epsilon \epsilon_{t-1} + (1-\rho_\epsilon) V_t$, $\epsilon_0\sim\textrm{Cauchy}$, $V_t\stackrel{iid}{\sim}\textrm{Cauchy}$, so the marginal distribution is $\epsilon_t\sim\textrm{Cauchy}$ for all $t$. 

\noindent \textbf{DGP 4}~
This DGP is for a log-linearized quantile Euler equation with over-identification, similar to the empirical application in \cref{sec:EIS}. 
The discount factor is $\beta_\tau=0.99$, the the EIS is $1/\gamma_\tau=0.2$. 
There are four excluded instruments, $Z_{j,t}=\rho_j Z_{j,t-1} + \eta_{j,t}$ for $j=1,2,3,4$, with $\rho_j=0.2j$. 
The vector $(\eta_{1,t},\eta_{2,t},\eta_{3,t},\eta_{4,t})$ is multivariate normal with unit variances and $\Corr(\eta_{j,t},\eta_{k,t})=0.1$ for $j \ne k$. 
Then, $X_t = \delta_0 + \delta_1 Z_{1,t} + V_t( \delta_2 Z_{2,t} + \delta_3 Z_{3,t} + \delta_4 Z_{4,t} ) + V_t$, and $Y_t = \ln(\beta_\tau)/\gamma_\tau + X_t/\gamma_\tau + U_t$. 
Error term vector $(U_t,V_t)$ is bivariate normal with standard deviations $\sigma_U=2$ and $\sigma_V=2$, and covariance $\sigma_{UV}=0.8$. 
Linking to \cref{sec:EIS}, $Y_t$ represents 
the log consumption ratio, and $X_t$ represents the log real interest rate. 

\subsection{Bandwidth selection}
\label{sec:app-comp-h}

For 
estimation, simply picking the smallest possible bandwidth seems reasonable and performs well in simulations. 
For linear models, the plug-in bandwidth from \citet{KaplanSun17} also seems to work well. 
In our experience, this ``optimal'' bandwidth from \citet{KaplanSun17} is usually much larger than the smallest numerically feasible value. 
For example, this is suggested by the figures in Section 7.3 of \citet{KaplanSun17}, where the plug-in optimal bandwidth is clearly well above the smallest fixed bandwidth they used. 
The same is true in their empirical example (Section 6), and it is again true in our empirical example, where the smallest feasible bandwidth is (almost) always below $0.0001$ while the optimal bandwidth is (usually) between $0.001$ and $0.05$.

For the estimator to attain asymptotic normality, with dependent data, it seems $nh_n^2\to\infty$ is required for \Cref{a:G-est}, which implies $h_n$ must be larger than $n^{-1/2}$. 
The smallest possible bandwidth is too small, but the bandwidth from \citet{KaplanSun17} satisfies this condition. 
%
Alternatively, one could experiment with a variation of the AMSE-optimal bandwidth (for estimating $\matf{G}$) proposed in \citet{Kato12}. 
With a linear model, this is easily done by replacing one of the $\vecf{X}$ in the rule-of-thumb formula in \citet{Kato12} with our $\vecf{Z}$, leaving the other $\vecf{X}$ as the regressor vector that in our notation includes the endogenous regressors in $\vecf{Y}$ and exogenous regressors in $\vecf{X}$. 
More careful consideration of inference-optimal bandwidths is left to future work (in progress).

\singlespacing
\bibliographystyle{ecta} 

\end{document}